\documentclass[a4paper,12pt]{article}
\usepackage{amsmath,amsthm,amsfonts,amssymb,color}
\usepackage[matrix,arrow]{xy}
\usepackage{graphicx}
\usepackage[UKenglish]{babel}
\usepackage{multicol} 
\usepackage{epstopdf}
\usepackage{hyperref}
\usepackage{comment}
\usepackage{caption}
\usepackage{subcaption}

\usepackage{multirow} 
\usepackage{float}
\usepackage{graphicx}
\usepackage{longtable}  

\usepackage{epstopdf}

\usepackage[left=2.5cm,right=2cm,top=2.5cm,bottom=3cm]{geometry}

\newtheorem{theorem}{Theorem}[section]

\newtheorem{remark}[theorem]{Remark}

\numberwithin{equation}{section}

\newcommand{\la}{\langle}
\newcommand{\ra}{\rangle}

\newcommand{\ut}{u(t)}
\newcommand{\xt}{x(t)}
\newcommand{\ft}{f(t)}

\newcommand{\us}{u(s)}
\newcommand{\xs}{x(s)}

\newcommand{\unt}{u_n(t)}
\newcommand{\xnt}{x_n(t)}
\newcommand{\fnt}{f_n(t)}

\newcommand{\uqt}{u_q(t)}
\newcommand{\xqt}{x_q(t)}
\newcommand{\fqt}{f_q(t)}

\newcommand{\uns}{u_n(s)}
\newcommand{\xns}{x_n(s)}

\newcommand{\until}{\widetilde{u}_n(t)}
\newcommand{\untils}{\widetilde{u}_n(s)}
\newcommand{\util}{\widetilde{u}(t)}

\newcommand{\wl}{\rightharpoonup}

\newcommand{\bx}{\mbox{\boldmath{$x$}}}
\newcommand{\by}{\mbox{\boldmath{$y$}}}
\newcommand{\bz}{\mbox{\boldmath{$z$}}}
\newcommand{\bff}{\mbox{\boldmath{$f$}}}
\newcommand{\bv}{\mbox{\boldmath{$v$}}}
\newcommand{\bu}{\mbox{\boldmath{$u$}}}

\newcommand{\btau}{\mbox{\boldmath{$\tau$}}}
\newcommand{\bepsilon}{\mbox{\boldmath{$\varepsilon$}}}
\newcommand{\bsigma}{\mbox{\boldmath{$\sigma$}}}

\newcommand{\bcero}{\mbox{\boldmath{$0$}}}
\newcommand{\bnu}{\mbox{\boldmath{$\nu$}}}
\newcommand{\bzero}{\mbox{\boldmath{$0$}}}

\newenvironment{keywords}{
  \vspace{2mm}
  \noindent
  \keywordsname: 
  \itshape\small
}

\newenvironment{mathsubclass}{
  \small
  \noindent
  \mathsubclassname: 
}

 \def\keywordsname{\textbf{Keywords}}
  \def\mathsubclassname{\textbf{2010 AMS Subject Classification}}

\newtheorem{teo}{Theorem}[section]

\newtheorem{defi}{Definition}[section]

\newtheorem{lema}{Lemma}[section]

\newtheorem{rem}{Remark}[section] 

\makeatletter
\renewcommand\@biblabel[1]{#1.}
\makeatother
 
\newcommand\be{\begin{equation}}
\newcommand\ee{\end{equation}}

\begin{document}

\title{ Optimal control of differential quasivariational inequalities with applications in contact mechanics}

\author{
Mircea Sofonea$^{1}$, Julieta Bollati$^{2}$, Domingo A. Tarzia$^{2}$ \\ \\
\small{{$^{1}$} \it D\'epartement de Math\'ematiques, Universit\'e de Perpignan Via Domitia,}\\
 \small{ \it   52 Avenue Paul Alduy, 66860 Perpignan, France}\\
\small{{$^2$} \it Depto. Matem\'atica - CONICET, FCE, Univ. Austral, Paraguay 1950} \\  
\small \it {S2000FZF Rosario, Argentina} \\
}
\date{}

\maketitle

\vskip 4mm

\abstract{\noindent We consider a  differential quasivariational inequality for which we state and prove the continuous dependence of the solution with respect to the data. This convergence result allows  us to prove the existence of at least one optimal pair  for an associated  control problem.  Finally, we illustrate our abstract results in the study of a free boundary problem  which describes the equilibrium of a viscoelastic body in frictionless contact with a foundation made of a rigid body coveblack by a rigid-elastic layer.

 \vskip 6mm

\vspace{0.03cm}

\vskip 6mm

\begin{mathsubclass}
35M87, 35R35, 47J20, 49J40, 49J45, 74M15.
\end{mathsubclass}

\vskip 18mm
\section{Introduction}\setcounter{equation}0

The present paper is motivated by  the study of mathematical models which describe the {\color{black} time-dependent} unilateral contact of a
deformable  body with a foundation.
Under appropriate mechanical assumptions on the constitutive law and the interface conditions, such kind of models lead to a weak formulation which is in the form of a system that couples an ordinary differential equation with a  variational or  quasivariational inequality.
Despite the fact that  the  solvability of such systems can be obtained by using  various abstract existence and uniqueness results available in the literature, at the best of our knowledge there  are very few results on the optimal control of the corresponding contact models.  In this current paper we try to fill this gap and, to this end, we use arguments of  variational and differential variational inequalities.

The theory of variational inequalities begun with the pioneering works \cite{St1964,LiSt1967,Br1972}. Later, various extensions and   applications were provided and the literature in the field is extensive. 
Comprehensive references on this subject are \cite{Ne2012,Li1969,EkTe1973,BaCa1984,KiSt1980,Cr1984,Ro1987,AlSo1993,DeMiPa2003}.
{\color{black}
A survey of several classes of time-dependent and evolutionary  variational inequalities, with our without   unilateral constraints, can be found in \cite{Gw2003}. There, results on existence and regularity for
parabolic and hyperbolic evolutionary variational inequalities can be found.}
The theory plays an important role in Mechanics, Physics and Engineering Sciences where a large number of free boundary problems lead to elliptic or parabolic variational inequalities problems. Some relevant examples of such problems are the free boundary problems related to fluid flows through porous media \cite{Ba1971}, phase-change processes for the one-phase Stefan problem \cite{Du1973} and two-phase Stefan problem \cite{Ta1979}.
Variational inequalities arise in the study of mathematical models  in Contact Mechanics too, as illustrated in the books \cite{DuLi1976,Pa1985,SoHaSh2006,SoMa2012,MiOcSo2013,SoMi2018}. Their optimal control
has been studied in \cite{Li1968,NeSpTi2006,HiPiUlUl2009,Tr2010,Cl2013}, for instance.

A differential variational inequality represents  a system that couples a differential equation with a variational or quasivariational inequality.  This terminology was used for the first time in \cite{AuCe1984}. 
Existence, uniqueness and convergence results have been obtained in \cite{Gw2013,LiZeMo2016,LiMiZe2017,LiZe2019}, {\color{black}for instance. 
A stability result for the solution set of differential variational inequalities  has been obtained in \cite{Gw2007,Gw2013-2}. There, perturbations of the associated set-valued mapping
and perturbations of the  set  of constraints have been consideblack. Moreover, the Mosco convergence of sets has been employed. The results in \cite{Gw2013-2} allow,  in particular, the treatment of quasistatic contact problems with short memory viscoelastic materials and Tresca's friction law.}
A new  class of differential quasivariational inequalities in Banach spaces has been consideblack in \cite{LiSo2018}. There, an existence and uniqueness result has been obtained by using a general fixed point principle. Moreover, some examples and applications have been presented, including  the variational analysis of a contact problem with viscoplastic materials. 

The current paper represents a continuation of \cite{LiSo2018}. Its aim is three fold. The first one is to complete the abstract existence and uniqueness result in  \cite{LiSo2018} with a general convergence result for the solution. Here we assume that all the problem data are perturbed, i.e., the second member and the initial condition of the differential equation, the monotone operator, the non-differentiable function, the convex set and the second member of the variational inequality, then we study the behaviour of the solution with respect these perturbations. The second aim  is to complete our previous work  \cite{LiSo2018} with
an existence result for an associated optimal control problem. Finally,  our third aim is to apply these new results in the study of an viscoelastic frictionless contact problem
with history-dependent hardening parameter.

The rest of the paper is structublack as follows. In Section \ref{s2} we introduce the differential quasivariational inequality we are interested in, denoted by
$\mathcal{P}$. Then, we recall some preliminary results which are needed  later in this paper. In Section \ref{s3} we present our general convergence result, Theorem \ref{t1}, which states the continuous dependence of the solution of  Problem $\mathcal{P}$ on the data.  The proof of the theorem is carried out in several steps, based on arguments on convexity, pseudomonotonicity and compactness.
Then, in Section \ref{s4} we introduce an optimal control problem associated to the  differential quasivariational inequality $\mathcal{P}$ and prove the existence of at least one optimal solution, Theorem \ref{t2}. Its proof is based on arguments of compactness and lower semicontinuity. Finally, in Section \ref{s5}, we present an application of our abstract results in the study of a mathematical model of contact with viscoelastic materials. We describe the model, list  the assumption on the data, then we state and prove  its unique weak solvability. Next, we prove the continuous dependence of the weak solution with respect to the data as well as the existence of the solution for an associated optimal control problem. We also provide the  mechanical interpretation of our results.

\section{Preliminaries}\label{s2}
\setcounter{equation}0

Throughout this paper  $I$ denotes either a bounded or an  unbounded time-interval, i.e., $I=[0,T]$ with $T>0$ or $I=\mathbb{R}_+=[0,+\infty)$. We consider two real Banach spaces $X$, $V$ and a real Hilbert space $Z$,
endowed with the inner product $(\cdot,\cdot)_Z$. The norm on these space will be denoted by $\Vert \cdot \Vert_X$, $\Vert \cdot \Vert_V$ and $\Vert \cdot \Vert_Z$, respectively.
The strong topological dual space of $V$ is denoted by $V^{*}$ and the duality paring of $V$ and $V^{*}$ is denoted by $\langle \cdot,\cdot \rangle$.
We shall use the symbols ``$\wl$'' and ``$\to$'' for the weak and strong convergence in various normed  spaces to be specified. All the limits, upper and lower limits are consideblack as $n\to\infty$, even if we do not mention it explicitly.
Moreover, we  use the notation $C(I;V)$ and $C(I;Z)$ for the space of continuous functions on $I$ with values in $V$ and $Z$, respectively. In addition, we denote by a dot above the derivative with respect to the time and we adopt the notation  $C^1(I;X)$ for the space of continuously differentiable function  defined on $I$ with values in $X$.

Consider the following data: $F:I\times X\times V\to X$, $x_0\in X$, $A:X\times V\to V^*$, $j:X\times V\times V\to \mathbb{R} $, $\pi:V\to Z$,  $f:I\to V$ and $K\subset V$. Then, the differential quasivariational inequality problem we consider in this paper is stated as follows.

\medskip

\noindent \textbf{Problem $\mathcal{P}$}. 
 {\it Find  $x\in C^1(I;X)$ and $u\in C(I;V)$ such that}
 \begin{align}
&\dot{x}(t)=F(t,x(t),u(t))\qquad \forall\, t\in I, \label{DiffEq-P}\\
&x(0)=x_0,\label{InitCond-P}\\
& u(t)\in K, \quad  \langle  A(\xt,\ut),v-\ut \rangle+j(\xt,\ut,v)-j(\xt,\ut,\ut) \nonumber\\
& \qquad \qquad \qquad\quad \geq (\ft,\pi v-\pi\ut)_Z\quad \forall\, v\in K,\  t\in I.  \label{VarIneq-P}
\end{align}

The study of Problem  $\mathcal{P}$ requires some preliminaries that we present in what follows.


\begin{defi} 
An operator $B:V\to V^*$ is said to be:
\begin{itemize} 
\item [{\rm(i)}] Lipschitz continuous, if there exists $L_B>0$ such that
 $$\| Bu_1-Bu_2\|_{V^*}\leq L_B \|u_1-u_2 \|_V \qquad \forall\,u_1,u_2\in V;$$
\item[{\rm(ii)}] strongly monotone, if there exists $m_B>0$ such that
 $$\la Bu_1-Bu_2,u_1-u_2\ra\geq m_B \|u_1-u_2 \|^2_V \qquad \forall\, u_1,u_2\in V.$$
\end{itemize}
\end{defi}

\medskip
Consider now the following assumptions on the data of Problem $\mathcal{P}$.
\begin{eqnarray}
&&\left\{ \begin{array}{ll} F:I\times X\times V\to X \mbox{\ is\ such that:}\\[2mm]
{\rm (a)\ The\ mapping\ } t\to F(t,x,u)\ {\rm is\  continuous\ for\ all} \ x\in X, u\in V.\\[2mm]
{\rm (b)\ For\ any\ compact\ set\ } J\subset I\ {\rm there\ exists\ } L_J>0 {\rm \ such\ that\ }\\[0mm]
\qquad \|F(t,x_1,u_1)-F(t,x_2,u_2)\|_X\leq L_J\, \left( \|x_1-x_2\|_X +\|u_1-u_2\|_V \right)\,\\[0mm]
\qquad {\rm for\ all\ }  x_1,x_2\in X,\; u_1, u_2\in V,\; t\in J.
\end{array}\right. \label{F}
\end{eqnarray}
\begin{equation}\label{x0}
x_0\in X.
\end{equation}
\begin{equation}\label{K}
K {\rm \ is\ a\ nonempty\ closed\ convex\ subset\ of\ } V.
\end{equation}
\begin{eqnarray}
&&\left\{ \begin{array}{ll} A:X\times V\to V^* \mbox{\ is\ such that:}\\[2mm]
{\rm (a)\ There\ exists\  } L'>0 {\rm \  such\ that\ }\\[0mm]
\qquad \|A(x_1,u)-A(x_2,u)\|_{V^*}\leq L' \|x_1-x_2 \|_X\\[0mm]
\qquad {\rm for\ all\ } x_1,x_2\in X,\; u\in V.\\[2mm]
{\rm (b)\ There\ exists\  } L''>0 {\rm \  such\ that\ }\\[0mm]
\qquad \|A(x,u_1)-A(x,u_2)\|_{V^*}\leq L'' \|u_1-u_2 \|_V\\[0mm]
\qquad {\rm for\ all\ } x\in X,\; u_1,u_2\in V.\\[2mm]
{\rm (c)\ There\ exists\  } m>0 {\rm \  such\ that\ }\\[0mm]
\qquad \la A(x,u_1)-A(x,u_2),u_1-u_2\ra\geq m \| u_1-u_2 \|^2_V\qquad\qquad \qquad\quad\\[0mm]
\qquad {\rm for\ all\ } x\in X,\; u_1,u_2\in V.
\end{array}\right. \label{A}
\end{eqnarray}
\begin{eqnarray}
&&\left\{ \begin{array}{ll} j:X\times V\times V\to \mathbb{R} \mbox{\ is\ such that:}\\[2mm]
{\rm (a)\ For\ all\  } x\in X {\rm \ and \ } u\in V, \;j(x,u,\cdot) {\rm \ is\ convex\ }\\[0mm]
\qquad {\rm and\ lower\ semicontinuous\ (l.s.c)\ on\  }V.\\[2mm]
{\rm (b)\ There\ exists\  } \alpha>0 {\rm \ and\ } \beta>0 {\rm \  such\ that\ }\\[0mm]
\qquad j(x_1,u_1,v_2)-j(x_1,u_1,v_1)+j(x_2,u_2,v_1)-j(x_2,u_2,v_2)\qquad\quad \\[0mm]
\qquad \leq \alpha \|x_1-x_2\|_X \|v_1-v_2 \|_V+\beta \|u_1-u_2\|_V \| v_1-v_2\|_V,\\[0mm]
\qquad {\rm for\ all\ } x_1,x_2\in X,\; u_1,u_2\in V, \; v_1,v_2\in V.
\end{array}\right. \label{j}
\end{eqnarray}
\begin{equation}
m>\beta. \label{m}
\end{equation}
\begin{equation}
f\in C(I;Z).\label{f}
\end{equation}
\begin{equation}\label{pi}
\left\{ \begin{array}{l} \pi:V\to Z\ {\rm is\ a\ linear\ continuous\ operator,\ i.e.,}\qquad\qquad\qquad\;\\[2mm]
{\rm there\ exists} \ c_0>0\ {\rm such\ that\ } \|\pi v\|_{Z} \leq c_0\,\|v\|_V\ \ \forall\,v\in V.
\end{array}\right.
\end{equation}
\medskip


Note  that assumption  (\ref{pi}) allows us to apply the Riesz representation theorem in order to define a function  $\overline{f}:I\to V^*$ such that
\begin{equation}
\la \overline{f},v\ra=\left(f(t),\pi v \right)_Z\qquad \forall\, v\in V,\; t\in I.\label{Riesz}
\end{equation}
Furthermore, assumption (\ref{f}) implies that $\overline{f}\in C(I;V^*)$.
Hence, the following results are obtained as a direct  consequence of Theorem 3.1 and  Lemma 3.6 in  \cite{LiSo2018}, respectively.


\begin{teo}\label{ExistUniq-P}
Assume that $X$ is a Banach space, $V$ is a reflexive Banach space, $Z$ is a Hilbert space and $(\ref{F})$--$(\ref{pi})$ hold. Then Problem $\mathcal{P}$ has a unique solution $(x,u)\in C^1(I;X)\times C(I;V)$.
\end{teo}

\begin{lema}\label{ExistUniq-VarIneq}
	Assume that  $X$ is a Banach space, $V$ is  a reflexive Banach space and $(\ref{K})$--$(\ref{pi})$ hold. Then, for each $\widetilde{x}(t)\in C^1(I;X)$, there exists a unique function $u\in C(I;V)$ such that
	\begin{align}
	u(t)\in K, \qquad& \langle  A(\widetilde{x}(t),\ut),v-\ut \rangle+j(\widetilde{x}(t),\ut,v)-j(\widetilde{x}(t),\ut,\ut)\nonumber\\
	&\qquad\geq (\ft,\pi v-\pi\ut)_Z,\qquad \forall\, v\in K,\ t\in I.  \label{VarIneq-xFija}
	\end{align}
\end{lema}

{\color{black} We now complete the previous results with  the following comments.

\begin{remark}\label{rem0}
	
Under the assumptions of Lemma \ref{ExistUniq-P} it is easy to see that the quasi\-variational inequality	
(\ref{VarIneq-xFija}) is equivalent with the problem of finding a function $u:I\to V$ such that
\begin{equation}\label{rr1}
u(t)\in K,\qquad  G(t, u,v)\ge 0\quad\ \forall\, v\in K,\ t\in I
\end{equation}
where $G:I\times K\times K\to \mathbb{R}$ is the function defined by
\begin{equation*}
G(t,u,v)= \langle A(\widetilde{x}(t),u),v-u \rangle+j(\widetilde{x}(t),u,v)-j(\widetilde{x}(t),u,u)
-(\ft,\pi v-\pi u)_Z
 \end{equation*}
for all $t\in I$, $ u,\,v\in K$. Let $t\in I$ be fixed. Then, it is easy to see that
$G(t,u,u)=0$ and $G(t,u,\cdot):K\to \mathbb{R}$ is a convex lower semicontinuous function, for any $u\in X$. Moreover, 
\[G(t,u,v)+G(t,v,u)\le -(m-\beta)\|u-v\|_V^2\le 0\qquad\forall\, u,\ v\in K.\]
All these properties allows us to use Theorem 1 in \cite{BlOe1994} in order to prove the solvability of the equilibrium problem
$(\ref{rr1})$.
For more details, existence results and applications of equilibrium problems, we refer to \cite{ChChRi1999,ChChRi2000} as well as to the edited volume \cite{DaGiMa2003}.
\end{remark}
	
}
\medskip

We end this section with the following  version of the Weierstrass theorem.

\begin{teo}\label{Weierstrass}
Let $W$ be a reflexive Banach space endowed with the norm $\|\cdot\|_W$, $U$ a weakly closed  subset of  $W$ and $J:U\to \mathbb{R}$ a weakly lower semicontinuous function. Then $J$ is bounded from below and attains its infimum on $U$ whenever one of the following two conditions hold:
\begin{itemize}
\item[{\rm(i)}] $U$ is bounded;
\item[{\rm(ii)}] $J$ is coercive, i.e., $J(p)\to \infty$\ \  as\ \  $\|p\|_W\to \infty$.
\end{itemize}
\end{teo}

We shall use Theorem \ref{Weierstrass} in Section \ref{s4} in order to  establish the existence of at least one solution of optimal control problem. Its proof can be found in many books and surveys,  including \cite{SoMa2012}.

\section{A convergence result}\label{s3}
\setcounter{equation}0

The solution $(x,u)$ to problem $\mathcal{P}$ obtained in Theorem \ref{ExistUniq-P} depends on the data  $F$, $x_0$, $A$, $K$, $j$ and $f$.
In this section we  prove a convergence result that shows the continuous dependence of $(x,u)$ with the above-mentioned data. This  result will represent  a crucial ingblackient in the study of the optimal control problem that we shall study in Section \ref{s4}. To describe it,
for each $n\in\mathbb{N}$ we consider a function $F_n$, an initial data $x_{0n}$, a convex set $K_n$, an operator $A_n$ and  two functions $j_n$ and $f_n$  that satisfy the assumptions (\ref{F})--(\ref{f}), respectively, with constants $L_{Jn}$, $L_n'$, $L_n''$, $m_n$, $\alpha_n$ and $\beta_n$. To avoid any confusion, when used with $n$, we shall refer to  these assumptions as (\ref{F})$_n$ --(\ref{f})$_n$.
The sequences $\{L_{Jn}\}$, $\{L_n'\}$, $\{L_n''\}$, $\{m_n\}$, $\{\alpha_n\}$ are assumed to be bounded  and, therefore, without the loss of generality we assume that
\begin{equation}
L_{Jn}\leq L_J,\quad\; L_n'\leq L',\quad \; L_n''<L''\quad \; m_n\geq m,\quad \;  \alpha_n\leq \alpha,\quad\; \beta_n\leq\beta\quad\forall\, n\in\mathbb{N} \label{Const-Conv}
\end{equation}
where $L_{J},L',L'',m,\alpha,\beta$ are the constants associated with the assumptions (\ref{F})--(\ref{f}), respectively.
Then, for each $n\in\mathbb{N}$ we consider the  following problem.

\medskip

\noindent \textbf{Problem $\mathcal{P}_n$}. 
{\it  Find $x_n\in C^1(I;X)$ and $u_n\in C(I;V)$ such that}
\begin{align}
&\dot{x}_n(t)=F_n(t,x_n(t),u_n(t))\qquad \forall\, t\in I, \label{DiffEq-Pn}\\
&x_n(0)=x_{0n},\label{InitCond-Pn}\\
& u_n(t)\in K_n, \qquad  \langle  A_n(\xnt,\unt),v_n-\unt \rangle\nonumber \\
& \qquad\qquad \qquad\qquad+j_n(\xnt,\unt,v_n) -j_n(\xnt,\unt,\unt)\nonumber \\
&\qquad\qquad\qquad\qquad\geq (\fnt,\pi v_n-\pi\unt)_Z\;\quad \forall\,v_n\in K_n,\ t\in I. \label{VarIneq-Pn}
\end{align}

Note that,  if (\ref{F})$_n$-(\ref{f})$_n$ and (\ref{pi}) hold, Theorem \ref{ExistUniq-P} guarantees the existence of  a unique solution  of problem $\mathcal{P}_n$, denoted in what follows by $(x_n,u_n)$. We now consider the following additional assumptions.
\begin{eqnarray}
&&\hspace{-32mm}\left\{ \begin{array}{ll} {\rm For\ all\ } n\in\mathbb{N} {\rm \ there\ exists\ } \Gamma_n \geq 0, {\rm \ and\ } \gamma_n\geq 0
{\rm \ such\ that:\ }\\[2mm]
{\rm (a)\ } \|F_n(t,x,u)-F(t,x,u) \|_X\leq \Gamma_n\left( \|x \|_X+ \|u\|_V+\gamma_n \right)\\[1mm]
\qquad {\color{black} \forall\, t\in I},
x\in X,\; u\in V.\quad\\[2mm]
{\rm (b)\ } \lim\limits_{n\to\infty} \Gamma_n=0.\\[3mm]
{\rm (c)\ }{\rm The\ sequence\  }  \lbrace \gamma_n\rbrace \subset \mathbb{R} {\rm \ is\  bounded.}\\[2mm]
\end{array}\right. \label{F-Conv}
\end{eqnarray}
\begin{equation}\label{x0n-Conv}
x_{0n}\to x_0 {\rm \ \ in\ }\ X.
\end{equation}
\begin{eqnarray}
&&\left\{ \begin{array}{ll} \lbrace K_n\rbrace {\rm \  converges\ to\ } K {\rm \ in\ the\ sense\ of\ Mosco\  \cite{Mo1969},\ i.e.,:\ } \\[2mm]
{\rm (a)\ }  {\rm For\ each\ } v\in K {\rm \ there\ exists\ a\ sequence\ } \lbrace v_n \rbrace {\rm \ such\ that}\qquad\quad\qquad\qquad\quad\\[2mm]
\qquad v_n\in K_n\ \:\forall\, n\in\mathbb{N} {\rm \ \ and\ }\  v_n\to v{\rm \ \ in\  }\ V.\\[2mm]
{\rm (b)\ } {\rm For\ each\ } \lbrace v_n\rbrace {\rm \ such\ that\ }\\[2mm]
\qquad v_n\in K_n\;\, \forall\, n\in\mathbb{N}\ {\rm \ and\ }\ v_n\wl v{\rm \ \ in\ }\ V,{\rm \ we\ have\ } {\color{black}v\in K}.
\end{array}\right. \label{Kn-Conv}
\end{eqnarray}

\begin{eqnarray}
&&\hspace{-3mm}\left\{ \begin{array}{ll} {\rm For\ all\ } n\in\mathbb{N} {\rm \ there\ exists\ } \Lambda_n \geq 0, {\rm \ and\ } \lambda_n\geq 0
{\rm \ such\ that:\ }\\[2mm]
{\rm (a)\ } \|A_n(x,u)-A(x,u) \|_{V^*}\leq \Lambda_n\left( \|x \|_X+ \|u\|_V+\lambda_n \right)\quad \forall\, x\in X,\; u\in V.\ \,\quad\\[2mm]
{\rm (b)\ }  \lim\limits_{n\to\infty} \Lambda_n=0.\\[2mm]
{\rm (c)\ }{\rm The\ sequence\  }  \lbrace \lambda_n\rbrace \subset \mathbb{R} {\rm \ is\  bounded.}\\[2mm]
\end{array}\right. \label{A-Conv}
\end{eqnarray}


\begin{eqnarray}
&&\left\{ \begin{array}{ll} {\rm (a)\ For\ all\ } n\in\mathbb{N} {\rm \ there\ exists\ } \tau_n\geq 0 {\rm \ and\ } \delta_n\geq 0{\rm \ such\ that:}\\[2mm]
\qquad j_n(x,u,v_1)-j_n(x,u,v_2)\leq \left[\tau_n+\delta_n( \|x \|_X+ \| u\|_V) \right] \|v_1-v_2 \|_V\\[2mm]
\qquad \quad \forall\, x\in X,\ u\in V,\ v_1,v_2\in V.\\[2mm]
{\rm (b) \ There\ exists\ }\tau_0>0{\rm \ and\ } \delta_0>0{\rm \ such\ that\ } \tau_n\leq \tau_0 {\rm \ and\ } \delta_n\leq \delta_0 <m.\\[2mm]
{\rm (c)}  {\rm \ For\ any\ sequences\ } \lbrace u_n \rbrace\subset V,\ \lbrace v_n \rbrace\subset V{\rm \ such\ that }\\[2mm]
\qquad  \; u_n\wl u {\rm \ in\ }  V, \; v_n\wl v {\rm \ in\ } V\quad\rm{\ we\ have\ }\\[2mm]
\qquad \limsup\limits_{n\to\infty} \left[j_n(x,u_n,v_n)-j_n(x,u_n,u_n) \right]\leq j(x,u,v)-j(x,u,u)\quad \forall\, x\in X.
\end{array}\right. \label{jn-Conv}
\end{eqnarray}

\begin{eqnarray}
&&\left\{ \begin{array}{ll} {\rm (a)\  } \fnt \wl \ft{\rm \ \ in\ }\ Z {\rm \ as\ }n\to\infty\quad \forall\,t\in I;\qquad\qquad\qquad\qquad\qquad\\[2mm]
{\rm (b)\ } {\rm For\ any\ compact\ set\ }J\subset I{\rm \ there\ exists\ } w_J>0{\rm \ such\ that}\qquad\qquad\;\;\quad\\[2mm]
\qquad \| f_n(t)\|_Z\leq w_J\quad \forall\, n\in\mathbb{N}, \ t\in J.
\end{array}\right. \label{fn-Conv}
\end{eqnarray}

\begin{eqnarray}
&&\left\{ \begin{array}{ll} {\rm \ For\ any\ sequence\  } \lbrace v_n\rbrace\subset V {\rm \ such\ that\ }\quad\qquad\qquad \qquad\qquad\qquad\qquad\qquad \\[2mm]
\qquad v_n \wl v{\rm \ \ in\ }\ V\ \ {\rm \ we \ have\ }\ \ \pi v_n\to\pi v{\rm \ \ in\ }\ Z.
\end{array}\right.\label{pi-Conv}
\end{eqnarray}

\medskip

Our main result of this section is the following.

\begin{teo}\label{t1}
	Assume  $(\ref{F})$--$(\ref{pi})$ and  $(\ref{F})_n$--$(\ref{f})_n$, for each $n\in\mathbb{N}$. Moreover, assume $(\ref{Const-Conv})$ and $(\ref{F-Conv})$--$(\ref{pi-Conv})$. Then, the  solution $(x_n,u_n)$ of Problem $\mathcal{P}_n$ converges   to the solution $(x,u)$ of Problem $\mathcal{P}$ as $n\to\infty$, i.e., for each $t\in I$ we have
	\begin{equation}\label{ez}
	\unt\to \ut\quad {\rm in\ }\ V\quad {\rm \ and \ }\quad \xnt\to\xt\quad {\rm in\ }\ X \qquad{\rm as\ }n\to\infty.
	\end{equation}
\end{teo}

The proof of Theorem \ref{t1} will be carried out in several steps. To present it, everywhere in what follows we assume that the hypotheses of   Theorem \ref{t1} are satisfied, even if we do not mention it explicitly. Moreover, for each $n\in\mathbb N$ we  consider the following  auxiliary problem in which, recall, $x\in C^1(I;X)$ is the first component of the solution
$(x,u)$ of Problem $\mathcal{P}$.

\medskip

\noindent \textbf{Problem $\mathcal{\widetilde{P}}_n$}. 
{\it  Find $\widetilde{u}_n \in C(I;V)$ such that}
\begin{align}
& \until\in K_n, \ \ \langle  A_n(\xt,\until),v_n-\until \rangle+j_n(\xt,\until,v_n)\label{VarIneq-PnTilde} -j_n(\xt,\until,\until) \nonumber\\
& \qquad \qquad \quad\qquad \geq (\fnt,\pi v_n-\pi\until)_Z\qquad \forall\,v_n\in K_n,\ t\in I. 
\end{align}

\medskip

The first step of the proof is the following.

\medskip

\begin{lema} \label{until-bounded} For each $n\in\mathbb{N}$, Problem $\mathcal{\widetilde{P}}_n$ has a unique solution $\widetilde{u}_n\in C(I;V)$. Moreover, for each  compact subset $J\subset I$, there exists $\widetilde{C}_J>0$ such that
\begin{equation}\label{xx}
\|\until\|_V\leq \widetilde{C}_J, \qquad \forall\,t\in J,\ n\in\mathbb{N}.
\end{equation}

\end{lema}

\begin{proof}
The existence and uniqueness of solution to problem $\mathcal{\widetilde{P}}_n$ is derived straightforward from  Lemma \ref{ExistUniq-VarIneq}.

Assume now that  $J\subset I$ is a given compact and let $t\in J$, $u_0\in K$. Using (\ref{Kn-Conv}) there exists a sequence $\lbrace u_{0n}\rbrace$ such that
$$u_{0n}\in K_n\quad \forall\, n\in\mathbb{N}\quad {\rm\ and\  } \quad u_{0n}\to u_0\  {\rm \ in\ \ }V.$$
Let $n\in\mathbb{N}$ be fixed and take $v_n=u_{0n}\in K_n$ in (\ref{VarIneq-PnTilde})
to obtain
\begin{align}
& \la  A_n(\xt,\until),\until-u_{0n} \ra\leq j_n(\xt,\until,u_{0n}) \nonumber\\
&\qquad-j_n(\xt,\until,\until)+ (\fnt,\pi\until-\pi u_{0n})_Z \nonumber
\end{align}
and, therefore,
\begin{align}
& \la  A_n(\xt,\until)-A_n(\xt,u_{0n}),\until-u_{0n} \ra \leq \la A_n(\xt,u_{0n}),u_{0n}-\until \ra  \nonumber\\
&\qquad + j_n(\xt,\until,u_{0n}) -j_n(\xt,\until,\until)+ (\fnt,\pi\until-\pi u_{0n})_Z. \nonumber
\end{align}
Then, using (\ref{A})$_n$(c) and  conditions (\ref{jn-Conv})(a) and (\ref{pi}) we find that
\begin{align}
& m_n \|\until-u_{0n} \|_V \leq \| A_n(\xt,u_{0n})\|_{V^*}   \nonumber\\
&\qquad+ \tau_n+\delta_n\left( \|\xt\|_X +\|\until\|_V \right)  + c_0\|\fnt\|_Z. \label{Anu0n-1}
\end{align}
Now, since
\begin{align*}
 \| A_n(\xt,u_{0n})\|_{V^*}&\leq \|A_n(\xt,u_{0n})-A(\xt,u_{0n}) \|_{V^*}\\
& +\|A(\xt,u_{0n})-A(\xt,u_0) \|_{V^*}+\|A(\xt,u_{0}) \|_{V^*}
\end{align*}
from assumptions (\ref{A-Conv}) and (\ref{A})(b) we obtain that
\begin{align}
\| A_n(\xt,u_{0n})\|_{V^*} &\leq  \Lambda_n\left(\|\xt \|_X +\| u_{0n}\|_V+\lambda_n \right)\nonumber\\
&+L'' \|u_{0n}-u_0 \|_V +\|A(\xt,u_{0}) \|_{V^*}. \label{Anuon-2}
\end{align}
Recall now  that conditions (\ref{Const-Conv}) and (\ref{jn-Conv})(b) guarantee that   $m_n\geq m$, $\tau_n\leq \tau_0$ and {\color{black}$\delta_n\leq \delta_0<m$}.  As $\|\until\|_V\leq \|\until-u_{0n}\|_V+\|u_{0n}\|_V$, combining  inequalities (\ref{Anu0n-1}) and (\ref{Anuon-2}) it follows that
\begin{align}
& \|\until-u_{0n}\|_V\leq \dfrac{1}{m-\delta_0} \Big\{\Lambda_n \left(\|\xt \|_X +\| u_{0n}\|_V+\lambda_n \right)+L'' \|u_{0n}-u_0 \|_V \nonumber\\
& \qquad +\|A(\xt,u_{0}) \|_{V^*}+ \tau_0+\delta_0\left( \|\xt\|_X +\|u_{0n}\|_V \right)+ c_0\|\fnt\|_Z  \Big\}.\label{cotauntilde}
\end{align}

\medskip

Next, since $u_{0n}\to u$, there exists $M>0$ which does not depend on $n$ such that
$\|u_{0n}-u_0 \|_V\leq M$. Consequently,
\begin{equation}\label{r1}
 \|u_{0n}\|_V\leq M+\| u_{0}\|_V.
 \end{equation}

On the other hand, from assumptions (\ref{A-Conv})(b),(c), we know that $\Lambda_n\to 0$ and $\{\lambda_n\}\subset\mathbb{R}$ in bounded. Therefore, there exists $\Lambda_0>0$ and $\lambda_0>0$ such that
\begin{equation}\label{r2}
\Lambda_n\leq \Lambda_0\qquad {\rm \ and \ }\qquad \lambda_n\leq \lambda_0. 
\end{equation}

In addition, since $x\in C^1(I;X)$, there exists $M_J>0$ which does not depend on $t$ such that
\begin{equation}\label{r3}
\|x(t)\|_X\leq M_J.
\end{equation}

Moreover, taking into account (\ref{A})(a) we get
\begin{align}
&\|A(\xt,u_{0}) \|_{V^*}\leq \|A(\xt,u_{0})-A(x_0,u_0) \|_{V^*}+\|A(x_0,u_{0}) \|_{V^*} \nonumber\\
&\qquad \leq L'\|\xt-x_0\|_X+\|A(x_0,u_{0}) \|_{V^*} \leq L'(M_J+\|x_0\|_X)+\|A(x_0,u_{0}) \|_{V^*}\label{r4}
\end{align}

Finally, from condition (\ref{fn-Conv})(b), there exists a constant $w_{J}>0$ which does not depend on $n$ and $t$ such that 
\begin{equation}\label{r5}
\|\fnt\|_Z\leq w_J.
\end{equation}

Therefore, from (\ref{cotauntilde})--(\ref{r5}) we deduce that 
\begin{align*}
& \|\until-u_{0n}\|_V\leq \dfrac{1}{m-\delta_0} \Big\{\Lambda_0 \left(M_J +M+\| u_{0}\|_V+\lambda_0 \right)+L'' M+L'(M_J+\|x_0\|_X) \nonumber\\
& \qquad +\|A(x_0,u_{0}) \|_{V^*}+ \tau_0+\delta_0\left( M_J +M+\|u_{0}\|_V \right)+ c_0w_J \Big\}.
\end{align*}
Defining now $C_J$ as the right hand side of the previous inequality we get that
$$\|\until\|_V\leq C_J+ \|u_{0n}\|_V.$$
As a result we deduce (\ref{xx}) with $\widetilde{C}_J=C_J+ M+\|u_{0}\|_V$,
which concludes the proof.
\end{proof}

The second step of the proof is the following.

\begin{lema}\label{Lem:until->u}  For each $t\in I$ the following weak convergence holds:
\begin{equation}\label{er}
\until\wl \ut\; {\rm \ \ in \ \ }V\  {\rm \ as \ }n\to\infty.
\end{equation}
\end{lema}

\begin{proof}
Let $t\in I$ and consider a compact set $J\subset I$ such that $t\in J$. Using Lemma \ref{until-bounded} we  obtain that there exists an element $\widetilde{u}(t)\in V$ and a subsequence of $\lbrace \until\rbrace$, still denoted by $\lbrace \until\rbrace$, such that $\until\wl \util$ in $V$ as $n\to\infty$. Recalling assumption (\ref{Kn-Conv}), since $\until\in K_n$\ $\forall\, n\in\mathbb{N}$, we deduce that $\util\in K$.

We now prove that $\util=\ut$
{\color{black} and, by the uniqueness of the solution
to (\ref{VarIneq-P}),
it is enough to show that $\util$ is a solution to inequality (\ref{VarIneq-P})}.
To this end we consider an element $v\in K$ and use (\ref{Kn-Conv}) to find  that there exists a sequence $\lbrace v_n\rbrace\subset V$ such that 
$$v_n\in K_n\quad \forall\, n\in\mathbb{N}\quad {\rm \ and \ }\quad v_n\to v\ \  {\rm \ in \ }V.$$
We now use (\ref{VarIneq-PnTilde}) to obtain
\begin{align}
 \la  A_n(\xt,\until),\until-v_n \ra&\leq j_n(\xt,\until,v_n) -j_n(\xt,\until,\until)\nonumber \\
& + (\fnt,\pi\until-\pi v_n)_Z. \nonumber
\end{align}
Next, writing
$${\color{black} A(\xt,\until)=A(\xt,\until)-A_n(\xt,\until)+A_n(\xt,\until)  }$$ we find that
\begin{align}
 \la  A(\xt,\until),\until-v_n \ra& \leq \la A_n(\xt,\until)-A(\xt,\until),v_n-\until\ra\nonumber\\
&+  j_n(\xt,\until,v_n) -j_n(\xt,\until,\until)\nonumber \\
& + (\fnt,\pi\until-\pi v_n)_Z.\nonumber
\end{align}
Adding and subtracting $v$ in the duality paring leads to
\begin{align}
 \la  A(\xt,\until),\until-v \ra&\leq \la  A(\xt,\until),v_n-v \ra \nonumber\\
 &+\la A_n(\xt,\until)-A(\xt,\until),v_n-\until\ra\nonumber\\
&+  j_n(\xt,\until,v_n) -j_n(\xt,\until,\until)\nonumber \\
& + (\fnt,\pi\until-\pi v_n)_Z. \label{A1}
\end{align}
So, 
\begin{equation} 
 \la  A(\xt,\until),\until-v \ra\leq \sum\limits_{i=1}^4 S_n^i(v_n),\label{sumandos}
\end{equation}
 with
\begin{equation} \label{defSni}
 \begin{array}{lll}
& S_n^1(v_n)=\la  A(\xt,\until),v_n-v \ra, &  \\ [1mm] 
&S_n^2(v_n)=\la A_n(\xt,\until)-A(\xt,\until),v_n-\until\ra, &\\ [1mm] 
&S_n^3(v_n)=   j_n(\xt,\until,v_n) -j_n(\xt,\until,\until),  &\\[1mm] 
& S_n^4(v_n)= (\fnt,\pi\until-\pi v_n)_Z.  &
 \end{array}
\end{equation}
In order to pass to the upper limit in inequality (\ref{sumandos}) we now estimate each of the terms $S_n^i$ above.

First, using (\ref{A})(b) we deduce that that
\begin{align}
S_n^1(v_n)&\leq \|A(\xt,\until)\|_{V^*} \| v_n-v\|_V\nonumber\\
&\leq \big( \|A(\xt,\until)-A(\xt,\util)\|_{V^*}+\|A(\xt,\util)\|_{V^*}\big) \| v_n-v\|_V\nonumber\\
& \leq \big(L'' \| \until-\util\|_V+\|A(\xt,\util)\|_{V^*} \big)\| v_n-v\|_V. \nonumber
\end{align}
Therefore, since $L'' \| \until-\util\|_V+\|A(\xt,\util)\|_{V^*}$ is bounded and $\| v_n-v\|_V\to 0$ it follows that
\begin{equation}
\limsup\limits_{n\to\infty} S_n^1(v_n)= \limsup\limits_{n\to\infty} \; \la  A(\xt,\until),v_n-v \ra \leq 0.\label{Sn1}
\end{equation}

Next, exploiting condition (\ref{A-Conv})(a)  we find that 
\begin{align}
S_n^2(v_n)&\leq  \| A_n(\xt,\until)-A(\xt,\until)\|_{V^*} \|v_n-\until\|_V \nonumber\\
&\leq  \Lambda_n \left( \|\xt\|_X+ \|\until\|_V+\lambda_n \right) \|v_n-\until\|_V.\nonumber
\end{align}
Taking now into account the boundedness of the sequences $\|v_n\|_V$, $\|\until\|_V$  and $\{\lambda_n\}$, using assumption (\ref{A-Conv})(b)  we obtain that
\begin{equation}
\limsup\limits_{n\to\infty} S_n^2(v_n)= \limsup\limits_{n\to\infty} \; \la A_n(\xt,\until)-A(\xt,\until),v_n-\until\ra \leq 0.\label{Sn2}
\end{equation}

We proceed with the term $S_n^3(v_n)$. From hypothesis (\ref{jn-Conv})(c), since $v_n\to v$ and $\until\wl\util$ in $V$ we have
\begin{align}
\limsup\limits_{n\to\infty} S_n^3(v_n)&= \limsup\limits_{n\to\infty}{\color{black}\big[j_n(\xt,\until,v_n) -j_n(\xt,\until,\until)\big]}\nonumber\\ 
&\leq j(\xt,\util,v)-j(\xt,\util,\util).\label{Sn3}
\end{align}

Finally, 
\begin{align}
S_n^4(v_n)&= (\fnt,\pi\until-\pi\util)_Z+(\fnt,\pi\util-\pi v)_Z+(\fnt,\pi v-\pi v_n)_Z\nonumber\\
&\leq \|\fnt\|_Z \|\pi \until-\pi\util\|_Z+(\fnt,\pi\util-\pi v)+\|\fnt\|_Z \|\pi v-\pi v_n\|_Z.\nonumber
\end{align}
Thus, by  assumptions (\ref{fn-Conv})(a) and (\ref{pi-Conv}), the weak convergences of $\until$ to $\util$ and the strong convergence of $v_n$ to $v$, both in $V$,  we deduce that
\begin{align}
&\limsup\limits_{n\to\infty} S_n^4(v_n)=\limsup\limits_{n\to\infty} \;(\fnt,\pi\until-\pi v_n)_Z\leq (\ft,\pi\util-\pi v)_Z.\label{Sn4}
\end{align}

We now  pass to the upper limit in inequality (\ref{sumandos}) and use (\ref{Sn1})--(\ref{Sn4}) to find that
\begin{align}
\limsup\limits_{n\to\infty}\;  \la  A(\xt,\until),\until-v \ra  & \leq j(\xt,\util,v)-j(\xt,\util,\util)\nonumber\\
&+(\ft,\pi\util-\pi v)_Z\qquad \forall\,v\in K.\label{LimSup}
\end{align}

{\color{black} On the other hand, using the monotonicity of the operator $A(x(t),\cdot)$ we have
\[\la A(\xt,v),\until-v \ra \le A(\xt,\until),\until-v \ra\qquad\forall\, v\in V\]
and, using the convergence  $\until\wl \util$ in $V$, we find that
\begin{equation}\label{zq}
\la A(\xt,v),\util-v \ra \le \limsup\limits_{n\to\infty}\;  \la  A(\xt,\until),\until-v \ra\qquad\forall\, v\in V.
\end{equation}
We now combine the inequalities (\ref{LimSup}) and (\ref{zq}) to deduce that
\begin{align}
\la A(\xt,v),\util-v \ra & \leq j(\xt,\util,v)-j(\xt,\util,\util)\nonumber\\
&+(\ft,\pi\util-\pi v)_Z\qquad \forall\,v\in K.\label{zs}
\end{align}
Consider now an arbitrary element $w\in K$  and let $\theta\in(0,1]$. We  take $v=\util+\theta (w-\util)$
in (\ref{zs}), use the convexity of the function $j$ with respect the third argument and divide the resulting inequality with $\theta>0$ to find that
\begin{align*}
\la A(\xt,\util+\theta (w-\util)),\util-w \ra & \leq j(\xt,\util,w)-j(\xt,\util,\util)\nonumber\\
&+(\ft,\pi\util-\pi w)_Z.
\end{align*}
We now pass to the limit as $\theta\to 0$ and use assumption (\ref{A})(b)
to conclude that $\util\in K$ satisfies the inequality
\begin{align}
\la A(\xt,\util),\util-w \ra& \leq j(\xt,\util,w)-j(\xt,\util,\util)\nonumber\\
&+(\ft,\pi\util-\pi w)_Z,\qquad \forall\,w\in K,\ t\in I.\label{Inecuntil}
\end{align}
On the other hand, Lemma \ref{ExistUniq-VarIneq} guarantees that (\ref{Inecuntil}) has a unique solution. } Therefore, (\ref{VarIneq-P}) and (\ref{Inecuntil}), yield 
$\util=\ut$.
This assertion  reveals that each subsequence of $\lbrace \until\rbrace$ which converges weakly in $V$ has the same limit $\ut.$ Therefore, by a standard argument we  get that the whole sequence $\lbrace \until\rbrace$ converges weakly to $\ut$ in $V$, which concludes the proof.
\end{proof}

\medskip

We now proceed with the following result.

\begin{lema} For each $t\in I$ the following strong convergence holds:
\begin{equation}\label{es}
\until\to\ut\; {\rm \ \ in \ \ }V {\rm \ \ as \ }\ n\to\infty.
\end{equation}

\end{lema}

\begin{proof}

Let $t\in I$ and let $J\subset I$ be a compact set such that $t\in J$. 
As $u(t)\in K$, assumption (\ref{Kn-Conv}) and arguments similar to those used in the proof of inequality (\ref{sumandos}) lead to
\begin{equation}\label{rr} 
\la  A(\xt,\until),\until-\ut \ra\leq \sum\limits_{i=1}^4 S_n^i(v_n).
\end{equation}
Here, for each $n\in\mathbb{N}$ and $i\in\{1,2,3,4\}$, $ S_n^i$ is given by (\ref{defSni}) and $\{v_n\}\subset V$ is a sequence such that
\begin{equation}
v_n\in K_n\;\; \ \forall\, n\in\mathbb{N}\;{\rm \ \ and\ }\ \; v_n\to \ut\;{\rm \ in \ \ }V.\label{seqtou}
\end{equation}
Inequality (\ref{rr}) implies that
\begin{align} 
& \la  A(\xt,\until)-A(\xt,\ut),\until-\ut \ra\nonumber\\
&\qquad\leq \la A(\xt,\ut),\ut-\until \ra +\sum\limits_{i=1}^4 S_n^i(v_n)\nonumber
\end{align}
and, using the strong monotonicity of $A$, (\ref{A})(c), yields
\begin{equation} 
m\| \until-\ut\|^2_{V}\leq \la A(\xt,\ut),\ut-\until \ra
 +\sum\limits_{i=1}^4 S_n^i(v_n).\label{CotaFuerte}
\end{equation}

On the other hand, the convergence (\ref{er}) in
Lemma \ref{Lem:until->u} implies that
\begin{equation}
\la A(\xt,\ut),\ut-\until \ra\to 0,\;{\rm\ as \ }\; n\to\infty.\label{Cota Axtut}
\end{equation} 
Moreover, using (\ref{Sn1})--(\ref{Sn4}), taking into account that $\until\wl \util=\ut$, replacing $v=\ut$ and considering the sequence $\lbrace v_n \rbrace$ such that (\ref{seqtou}) holds, we see that
\begin{equation}
\lim\sup\limits_{n\to\infty} \sum\limits_{i=1}^4 S_n^i(v_n)\leq 0.\label{CotaLimSupSumandos}
\end{equation}
Therefore, passing to the upper limit in (\ref{CotaFuerte}) and using (\ref{Cota Axtut}), (\ref{CotaLimSupSumandos}) we deduce  that
$$\lim\sup_{n\to\infty} \| \until-\ut\|^2_V\leq 0,$$ which implies (\ref{es}).
\end{proof}

We are now in a position to provide the proof of Theorem \ref{t1}.

\begin{proof}

Let $t\in I$ and $n\in\mathbb{N}$. Moreover, consider a compact interval $J\subset I$ such that $[0,t]\subset J$  and denote by $L_J$ the constant which arises in condition (\ref{F})(b).
We test with $v_n=\until\in K_n$ in {\color{black}(\ref{VarIneq-Pn})}  to see that
\begin{align}
 \la  A_n(\xnt,\unt),\unt-\until \ra&\leq j_n(\xnt,\unt,\until) -j_n(\xnt,\unt,\unt)\nonumber \\
& + (\fnt,\pi\unt-\pi \until)_Z. \label{PrimerTermino}
\end{align}
Then, taking $v_n=\unt\in K_n$ in  {\color{black}(\ref{VarIneq-PnTilde})} we find that
\begin{align}
 \la  A_n(\xt,\until),\until-\unt \ra&\leq j_n(\xt,\until,\unt) -j_n(\xt,\until,\until)\nonumber \\
& + (\fnt,\pi\until-\pi \unt)_Z. \label{SegundoTermino}
\end{align}
We now add inequalities (\ref{PrimerTermino}) and (\ref{SegundoTermino}) to deduce that
\begin{align}
&  \la  A_n(\xnt,\unt) -A_n(\xt,\until) ,\unt-\until\ra\leq j_n(\xnt,\unt,\until) \nonumber \\
& \qquad-j_n(\xnt,\unt,\unt)+ j_n(\xt,\until,\unt) -j_n(\xt,\until,\until).\nonumber
\end{align}
Next, writing
\begin{align*}
 A_n(\xnt,\unt)-A_n(\xt,\until)&=A_n(\xnt,\unt)-A_n(\xnt,\until) \nonumber\\
&+ A_n(\xnt,\until)-A_n(\xt,\until),
\end{align*}
we get
\begin{align}
& \la  A_n(\xnt,\unt)-A_n(\xnt,\until),\unt-\until \ra\nonumber\\
&\qquad \leq \la  A_n(\xt,\until)-A_n(\xnt,\until),\unt-\until \ra \nonumber\\
&\qquad + j_n(\xnt,\unt,\until) -j_n(\xnt,\unt,\unt) \nonumber\\
 &\qquad +j_n(\xt,\until,\unt) -j_n(\xt,\until,\until).\nonumber 
\end{align}
Therefore, using assumptions (\ref{A})$_n$(c) and (\ref{j})$_n$(a) we obtain that
\begin{align}\label{za}
&m_n\|\unt-\until\|^2_V\nonumber \leq \| A_n(\xt,\until)-A_n(\xnt,\until)\|_{V^*}\|\unt-\until \|_V\nonumber\\
&\qquad +\alpha_n \|\xnt-\xt\|_X \|\until-\unt \|_V+\beta_n\|\unt-\until\|^2_V.
\end{align}
Next, {\color{black} assumptions (\ref{A})$_n$(a), (\ref{Const-Conv}) and inequality (\ref{za}) imply that}
\begin{equation}
\|\unt-\until \|_V\leq \tfrac{L'+\alpha}{m-\beta} \| \xt-\xnt\|_X.\label{Dif:unt-until}
\end{equation}
Therefore, {\color{black} from (\ref{Dif:unt-until})} we deduce that
\begin{align}
\|\unt-\ut \|_V &\leq \|\unt-\until \|_V+\|\until-\ut \|_V  \nonumber\\
&\leq \tfrac{L'+\alpha}{m-\beta} \| \xt-\xnt\|_X+ \|\until-\ut \|_V.\label{Dif:unt-ut}
\end{align}

On the other hand, since $\xt$ and $\xnt$ satisfy (\ref{DiffEq-P})--(\ref{InitCond-P}) and (\ref{DiffEq-Pn})--(\ref{InitCond-Pn}), res\-pectively, we find  that
\begin{eqnarray*}
&&x(t)=x_0+\int_0^t F(s,\xs,\us)\; ds,\\ [2mm]
&&\xnt=x_{0n}+\int_0^t F_n(s,\xns,\uns)\;ds
\end{eqnarray*}
and, therefore,
\begin{eqnarray}\label{Cota:xt-xnt}
&&\|\xt-\xnt\|_X\leq \| x_0-x_{0n}\|_X\nonumber \\
&&\qquad\qquad+\int_0^t \| F(s,\xs,\us)-F_n(s,\xns,\uns)\|_X \;ds.
\end{eqnarray}
Now, using (\ref{F})$_n$(b)  and (\ref{F-Conv})  we obtain that
\begin{align}
& \| F(s,\xs,\us)-F_n(s,\xns,\uns)\|_X\leq \| F(s,\xs,\us)-F_n(s,\xs,\us)\|_X\nonumber\\
&\qquad\qquad+\| F_n(s,\xs,\us)-F_n(s,\xns,\uns)\|_X\nonumber\\
&\qquad \qquad\leq  \Gamma_n\left( \|\xs\|_X+\| \us\|_V+\gamma_n  \right)\nonumber\\
&\qquad \qquad + L_J \left( \|\xs-\xns \|_X+\|\us-\uns \|_V \right).\label{Cota:xt-xnt-1}
\end{align}

 We combine  (\ref{Dif:unt-ut}) and (\ref{Cota:xt-xnt-1})  to find that
\begin{align}
& \| F(s,\xs,\us)-F_n(s,\xns,\uns)\|_X\leq  \Gamma_n\left( \|\xs\|_X+\| \us\|_V+\gamma_n  \right)\nonumber\\
&\qquad + L_J\left(1+\tfrac{L'+\alpha}{m-\beta} \right) \|\xns-\xs\|_X+ L_J \| \untils-\us\|_V.\label{Cota:Fn-F}
\end{align}
Then, exploiting (\ref{Cota:xt-xnt}) and taking into account  (\ref{Cota:Fn-F})  we deduce that
\begin{align}
\|\xt-\xnt \|_X\leq g_n(t)+c\int_0^t \| \xs-\xns\|_X\; ds,
\end{align}
with $c=L_J\left(1+\tfrac{L'+\alpha}{m-\beta} \right)$ and
\begin{align}
g_n(t)&=\| x_0-x_{0n}\|_X+\int_0^t  \Gamma_n\left( \|\xs\|_X+\| \us\|_V+\gamma_n  \right)  \; ds\nonumber\\
&+\int_0^t L_J \| \untils-\us\|_V \; ds. 
\end{align}
We now use the Gronwall argument to see that  
\begin{align}\label{azz}
\|\xt-\xnt \|_X\leq g_n(t) \; e^{ct}.
\end{align}
Moreover, note that assumptions (\ref{x0n-Conv}), (\ref{F-Conv}), {\color{black} the bound (\ref{xx})} and the convergence (\ref{es}) allow us
to use the Lebesgue
dominated convergence theorem  to obtain  that
\begin{equation}\label{az}
g_n(t)\to 0\qquad {\rm \ as\ }\qquad n\to\infty.
\end{equation}
We now use (\ref{azz})  and (\ref{az}) to see that
$\xnt\to \xt$ in $X$. Then,  (\ref{Dif:unt-ut}) implies $\unt\to \ut$ in $V$, which concludes the proof.
\end{proof}

We end this section with the following remarks.

\begin{rem}\label{rem1} Assume that
\begin{equation}\label{ffq}
\theta\in C(I;\mathbb{R}),\quad \widetilde{f}_n\in Z\quad{\rm and}\quad f_n(t)=\theta(t) \widetilde{f}_n\qquad \forall\, n\in\mathbb{N},\  t\in I.
\end{equation}	
In addition, assume that 
\begin{equation} 
\widetilde{f}\in Z,\quad f(t)=\theta(t) \widetilde{f}\qquad \forall\, t\in I
\quad{\rm and}\quad
	\widetilde{f}_{n}\wl \widetilde{f}\quad{\rm \ in\ }Z.\label{wide}
	\end{equation}
	Then it is easy to check that $(\ref{f})$, $(\ref{f})_n$ and 
	$(\ref{fn-Conv})$ hold and, therefore, 
	the statement of Theorem $\ref{t1}$ still remains valid if we replace these assumptions
	by hypotheses $(\ref{ffq})$ and $(\ref{wide})$.
\end{rem}

{\color{black} \begin{rem}\label{re2}  Note that
Theorem $\ref{t1}$ provides a pointwise convergence result for the solution  $(x_n,u_n)$ of Problem $\mathcal{P}_n$  to the solution $(x,u)$ of Problem $\mathcal{P}$ as $n\to\infty$, see $(\ref{ez})$.  Extending this  result to a convergence  result in the space $C^1(I;X)\times C(I;V)$ remains  on open problem which deserves to be investigated in the future.\end{rem}}

\section{An optimal control problem}\label{s4}
\setcounter{equation}0

Throughout  this section we assume that $(W,\|\cdot\|_W)$ is a reflexive Banach space and $U$ is a nonempty subset of $W$.
For each $q\in U$ we consider a function $F_q$, an initial data $x_{0q}$, a convex set $K_q$, an operator $A_q$ and  two functions $j_q$ and $f_q$  that satisfy the assumptions (\ref{F})--(\ref{f}), respectively with constants $L_{Jq}$, $L_q'$, $L_q''$, $m_q$, $\alpha_q$ and $\beta_q$. To avoid any confusion, when used with $q$ we will  refer to these assumptions as (\ref{F})$_q$--(\ref{f})$_q$.
We now  consider the following problem.

\medskip

\noindent \textbf{Problem $\mathcal{P}_q$}. 
{\it Find $x_q\in C^1(I;X)$ and $u_q\in C(I;V)$ such that}
\begin{align}
&\dot{x}_q(t)=F_q(t,x_q(t),u_q(t))\qquad \forall\,t\in I, \label{DiffEq-Pq}\\
&x_q(0)=x_{0q},\label{InitCond-Pq}\\
& \uqt\in K_q, \qquad  \langle  A_q(\xqt,\uqt),v_q-\uqt \rangle\nonumber \\
& \qquad \quad \qquad\qquad +j_q(\xqt,\uqt,v_q)-j_q(\xqt,\uqt,\uqt)
\nonumber \\
& \qquad \qquad \qquad\quad\geq (\fqt,\pi v_q-\pi\uqt)_Z\quad\forall\, v_q\in K_q,\ t\in I. \label{VarIneq-Pq} 
\end{align}
Under assumptions (\ref{F})--(\ref{f}), (\ref{pi}), {\color{black} Theorem \ref{ExistUniq-P}} guarantees that for each $q\in U$ there exists a unique solution $(x_q,u_q)\in C^1(I;X)\times C(I;V)$ to Problem $\mathcal{P}_q$}.

Consider now a cost function $\mathcal{L}:X\times V\times U\to\mathbb{R}$. Then, the optimal control problem we study in this section is the following.
\medskip

\noindent \textbf{Problem $\mathcal{Q}$}. 
{\it  Given $t\in I$, find $q^*\in U$ such that} 
\begin{align}\label{Problem:Q}
\mathcal{L}(x_{q^*}(t),u_{q^*}(t),q^*)=\min\limits_{q\in U}\ \mathcal{L}(\xqt,\uqt,q).
\end{align}

In the study of this problem we consider the following assumptions.
\begin{eqnarray}
U {\rm\ is\ a\ nonempty\ weakly\ closed\ subset\ of\ }W. \label{U}
\end{eqnarray}
\begin{eqnarray}
&&\left\{ \begin{array}{ll} {\rm For\ all\ sequences\ }  \lbrace x_n \rbrace\subset X,\; \lbrace u_n\rbrace\subset V,\;\lbrace q_n\rbrace\subset U{\rm\  such\ that\ }\\[2mm]
x_n\to x{\rm\ in\ }X,\; u_n\to u{\rm\ in\ }V,\; q_n\wl q{\rm\ in\ }W,\; {\rm\ we\ have\ }\\[2mm]
\liminf\limits_{n\to\infty} \mathcal{L}(x_n,u_n,q_n)\geq \mathcal{L}(x,u,q).
\end{array}\right. \label{LimInfL}
\end{eqnarray}
\begin{eqnarray}
&&\left\{ \begin{array}{ll} {\rm There\ exists\ } z:U\to\mathbb{R}{\rm\ such\ that}\\[2mm]
{\rm (a)\ } \mathcal{L}(x,u,q)\geq z(q)\quad \forall \,x\in X,\; u\in V,\; q\in U.\hspace{25mm} \\[2mm]
{\rm (b)\ } \|q_n\|_W\to\infty {\rm\ implies\ that\ } z(q_n)\to\infty.\qquad
\end{array}\right. \label{CoerL}
\end{eqnarray}

\begin{equation}
U {\rm\ is \ a \ bounded\ subset\ of\ }W.\label{Ubound}
\end{equation}

\medskip
Our main  result of this section is the following. 

\begin{teo}\label{t2}
Assume  $(\ref{F})_q$--$(\ref{f})_q$,  for each $q\in U$. In addition, assume $(\ref{pi})$, $(\ref{pi-Conv})$, $(\ref{U})$, $(\ref{LimInfL})$ and either $(\ref{CoerL})$ or $(\ref{Ubound})$. For each sequence $\lbrace q_n\rbrace\subset U$ such that $q_n\wl q$ in $W$ define 
$$F=F_{q},\quad x_{0}=x_{0q},\quad K=K_{q},\quad A=A_{q},\quad j=j_{q},\quad f=f_{q}$$
and
$$F_n=F_{q_n},\quad x_{0n}=x_{0q_n},\quad K_n=K_{q_n},\quad A_n=A_{q_n},\quad j_n=j_{q_n},\quad f_{n}=f_{q_n}$$
and assume that  $(\ref{Const-Conv})$, $(\ref{F-Conv})$--$(\ref{fn-Conv})$ hold.
Then, for each $t\in I$, the optimal control problem $\mathcal{Q}$ has at least one solution $q^*$.
\end{teo}

\begin{proof}
Let $t\in I$ be fixed and consider the function $J_t:U\to\mathbb{R}$ defined by
\begin{equation}\label{Jt}
J_t(q)=\mathcal{L}(\xqt,\uqt,q)\qquad \forall\, q\in U.
\end{equation}
Then, we consider  the problem of finding $q^*$ such that
\begin{equation}\label{MinJt}
J_t(q^*)=\min\limits_{q\in U} J_t(q).
\end{equation}
We apply Theorem \ref{t1}  to see that $x_{q_n}(t) \to x_q(t)$ in $X$ and $u_{q_n}(t)\to u_q(t)$ in $V$. 
Then, taking into account the convergence $q_n\wl q$ in $U$,   the definition (\ref{Jt}) of $J_t$  and condition (\ref{LimInfL}) on $\mathcal{L}$ we find that
\begin{align}
\liminf_{n\to\infty} J_t(q_n)=\liminf_{n\to\infty} \mathcal{L}(x_{q_n}(t),u_{q_n}(t),q_n)\geq \mathcal{L}(x_q(t),u_q(t),q)=J_t(q).
\end{align}
This means that $J_t$ is a weakly lower semicontinuous function. 

Assume now that  condition (\ref{CoerL}) is satisfied. Then 
$$J_t(q_n)=\mathcal{L}(x_{q_n},u_{q_n},q_n)\geq z(q_n)$$
and  $\|q_n\|_W\to \infty$ implies $z(q_n)\to\infty$. It follows from here that $J_t(q_n)\to\infty$, i.e., $J_t$ is coercive.
Recalling that $W$ is a reflexive Banach space and $U$ is a weakly closed subset of $W$, the existence of at least one solution to problem (\ref{MinJt}) is a direct consequence of Theorem \ref{Weierstrass}. This means that there exists a minimizer $q^*\in U$ for  $J_t$ which, in turn,  guarantees that Problem $\mathcal{Q}$ has at least one solution. The same conclusions follows if we assume that  condition (\ref{Ubound}) is satisfied since, in this case,  the Weierstrass-type argument provided by Theorem \ref{Weierstrass} still holds.
\end{proof}

We end this section with the following remark.

\begin{rem}\label{rem2} Assume that
	\begin{equation}\label{fq}
	\theta\in C(I;\mathbb{R}),\quad \widetilde{f}_q\in Z\quad{\rm and}\quad f_q(t)=\theta(t) \widetilde{f}_q\qquad \forall\, Q\in U,\  t\in I.
	\end{equation}	
	In addition, assume that 
	\begin{equation}\label{wid}
	\widetilde{f}_{q_n}\wl \widetilde{f}_q\quad{\rm \ in\ }Z\ \ \mbox{\rm for any sequence $\lbrace q_n\rbrace\subset U$ such that $q_n\wl q$\ in\ $W$.}
	\end{equation}
	Then it is easy to check that
	$(\ref{f})_q$
	and $(\ref{fn-Conv})$ hold and, therefore, 
	the statement of Theorem $\ref{t2}$ still remains valid if we replace
	these assumptions  by hypotheses $(\ref{fq})$, $(\ref{wid})$.
\end{rem}

\section{A frictionless contact problem}
\label{s5}\setcounter{equation}0

As {\color{black} mentionned} in the Introduction, the results in Section \ref{s3}--\ref{s4}  can be used in the analysis and control of mathematical models which describe the contact of a deformable body with a foundation.  A large number of examples can be consideblack, in which the contact is frictional or frictionless and the material behaviour  is described by an elastic, viscoelastic or {\color{black} viscoplastic} constitutive law. In this section  we provide such an example in which we assume that the contact is frictionless, the material is viscoelastic and the hardening of the foundation is taken into account.  For more details on  the modelling and analysis of contact problems we   refer the reader to the books \cite{SoMa2012}, \cite{SoMi2018}.

Everywhere below $d\in\{2,3\}$, $\mathbb{S}^d$ denotes  the space of second order symmetric tensors on $\mathbb{R}^d$  and $``\cdot"$, $\|\cdot\|$
will represent the inner product and the Euclidean norm on $\mathbb{R}^d$ and $\mathbb{S}^d$, respectively. We use the notation $\bcero$ for the zero element of the spaces  $\mathbb{R}^d$ and $\mathbb{S}^d$
and the indices $i,j, k, l$ run from 1 to $d$. 
Let $\Omega\subset\mathbb{R}^d$ be a bounded domain with Lipschitz continuous boundary $\Gamma$ {\color{black}and let $\overline\Omega=\Omega\cup\Gamma$}. We denote by $\bnu$  the outward unit normal at $\Gamma$ and   $\by\in\Omega\cup\Gamma$  will represent the spatial variable which, sometimes, for simplicity, is skipped.
Assume that $\Gamma=\Gamma_1\cup \Gamma_2\cup\Gamma_3$ where $\Gamma_1$, $\Gamma_2$, $\Gamma_3$ are mutually disjoint  measurable parts of $\Gamma$ such that ${ meas}\;(\Gamma_1)>0$.  For the displacement and the stress field we use the Hilbert spaces $(V,(\cdot,\cdot)_V)$ and $(Q,(\cdot,\cdot)_Q)$, respectively, defined by
\begin{align*}
&V=\lbrace\, \bv=(v_i)\in H^1(\Omega)^d\;:\; \bv\vert_{\Gamma_1}=\bcero\,\rbrace,\quad
(\bu,\bv)_V=\int_\Omega \bepsilon(\bu)\cdot \bepsilon(\bv)\,dy,
\\
&Q=\lbrace\, \btau=(\tau_{ij})\in L^2(\Omega)^{d\times d}\;:\; \tau_{ij}=\tau_{ji}\,\rbrace,
\quad (\bsigma,\btau)_Q=\int_\Omega
\bsigma\cdot \btau\,dy.\end{align*}
Here and below $\bepsilon$ represents the deformation operator, i.e.,  $\bepsilon(\bv)$ denotes the symmetric part of the {\color{black} gradient of  $\bv$}, for any $\bv\in V$. The associate norms on the spaces $V$ will be denoted by $\|\cdot\|_V$ and $\|\cdot\|_Q$, respectively.

For an element $\bv\in V$, we use the notation $v_\nu$ and $\bv_\tau$ for the normal and tangential traces of 
$\bv$ on $\Gamma$, i.e., $v_\nu=\bv \cdot \bnu$ and $\bv_{\tau}=\bv-v_\nu \bnu$. Moreover, for a regular stress field $\bsigma\in Q$ we use the notation $\sigma_\nu=(\bsigma \bnu)\cdot \bnu$ and $\bsigma_\tau={\color{black} \bsigma-\sigma_\nu \bnu}$. Finally,
as usual, we denote by $V^*$ the strong topological dual of $V$, by $\la \cdot,\cdot\ra$ the duality paring mapping and by $I$ an interval of time of the form $I=[0,T]$ with $T>0$ or $I=[0,+\infty)$.

Then, the classical formulation of the viscoelastic contact problem we consider in this section is the following.
\medskip

\noindent \textbf{Problem $\mathcal{P}^{ve}$}. 
{\it Find a stress field $\bsigma:\Omega\times I\to \mathbb{S}^d$, a displacement field $\bu:\Omega\times I\to\mathbb{R}^d$ and an interface function $\eta_\nu:\Gamma_3\times I\to
\mathbb{R}$ such that}
\begin{align}
\dot{\bsigma}(t)=\mathcal{E}\bepsilon(\dot{\bu}(t))+\beta(\bsigma(t)-\mathcal{F}(\bepsilon(\bu(t))) &\qquad {\rm in\ } \Omega,\label{S58}\\
{\rm Div\ }\bsigma(t)+\bff_0(t)=\bcero &\qquad {\rm in\ }\Omega,\label{S59}\\
\bu(t)=\bcero &\qquad {\rm on\ }\Gamma_1,\label{S60}\\
\bsigma(t)\cdot \bnu=\bff_2(t)&\qquad {\rm on\ }\Gamma_2,\label{S61}
\end{align}
\begin{align}
\left.\begin{array}{lll}
u_\nu(t)\leq g,\quad \sigma_\nu(t)+ku_\nu^+(t)+\eta_\nu(t)\leq 0, \\ [2mm]
(u_\nu(t)-g)(\sigma_\nu(t)+ku_\nu^+(t))+\eta_\nu(t))=0,\\ [2mm]
0\le\eta_\nu(t)\le h\Big(\displaystyle\int_0^tu_\nu^+(s)\,ds,{\color{black}u^+_\nu(t)}\Big),
\\[3mm]
{\color{black}\eta_\nu(t)=}\left\{\begin{array}{ll}
{\color{black}0\ \hspace{38mm} {\rm if}\ \ u_\nu(t)<0}, \\[4mm]
{\color{black}h\Big(\displaystyle\int_0^tu_\nu^+(s)\,ds,u^+_\nu(t)\Big)\ \ \,\,{\rm if}\ \ u_\nu(t)>0} \end{array}
\right.\end{array}
\right\}&\qquad {\rm on\ }\Gamma_3,\label{S62}\\
\bsigma_\tau(t)=\bcero& \qquad {\rm on\ }\Gamma_3,\label{S63}\\
\bsigma(0)=\bsigma_0,\qquad \bu(0)=\bu_0 &\qquad {\rm in\ }\Omega.\label{S64}
\end{align}

Note that Problem $\mathcal{P}^{ve}$ describes the equilibrium of a viscoelastic body which occupies the domain $\Omega$,  is held fixed on the part $\Gamma_1$ on his boundary, is acted upon by a  time-dependent surface traction of density $\bff_2$ on $\Gamma_2$ and is in contact with a foundation on $\Gamma_3$.
Equation (\ref{S58}) represents the constitutive law which models the viscoelastic behavior of the material. Here  
$\mathcal{E}$ is a fourth order elasticity tensor, $\beta$ is a viscosity coefficient  and  $\mathcal{F}$ is a  constitutive function.
Equation (\ref{S59}) represents the equilibrium equation in which $\bff_0$ denotes the density of body forces, (\ref{S60}) is the displacement boundary condition and (\ref{S61}) is the traction boundary condition.

Condition (\ref{S62}) is the contact condition  which models the contact with a foundation made of a rigid body coveblack by a layer of rigid-elastic material. Here $g$ represents the thickness of this layer, $h$ is a given function which describes its rigidity, $k$ is a stiffness coefficient and $r^+$ denotes the positive part of $r$, i.e., $r_+={\rm max}\,\{r,0\}$.
Details can be found in \cite{SoMi2018}. Here we restrict ourselves to recall that the quantity
\begin{equation}\label{csi}
\xi(\by,t)=\int_0^tu_\nu^+(\by,s)\,ds
\end{equation}
represents the  accumulated penetration in the point $\by$ of the contact surface at the time moment $t$. Assuming that the yield function $h$ depends on the 
{\color{black} process variables $\xi$ and $u^+_\nu$}
describes  the hardening  property of the foundation.

Condition (\ref{S63}) shows that the 
tangential component of the stress vanishes on the contact surface  and, therefore, the contact is frictionless. Finally, (\ref{S64}) are the initial conditions, in which $\bu_0$ and $\bsigma_0$ are given.

In the study of Problem $\mathcal{P}^{ve}$  
we use the space of symmetric fourth order tensors ${\bf Q_\infty}$ defined by
${\bf Q_\infty}=\{\, {\cal C}=(c_{ijkl})\mid 
{c}_{ijkl}={c}_{jikl}={c}_{klij} \in L^\infty(\Omega)\}$ and we consider the  following assumption on the data. 
\begin{eqnarray}\label{E}
&&\left\{ \begin{array}{ll} 
{\rm (a)\ } \mathcal{E}\in{\bf Q_\infty}.\\[2mm]
{\rm (b)\ } {\rm \ There\ exists\ }m_\mathcal{E}>0{\rm \ such\ that\ }\\
\qquad \mathcal{E}(\by)\btau\cdot\btau\geq m_\mathcal{E}\|\btau\|^2\quad{\rm for\ all\ }\btau\in\mathbb{S}^d, {\rm\ a.e.}\ \, \by\in\Omega.\hspace{5mm}
\end{array}\right.
\end{eqnarray}
\begin{eqnarray}
&&\hspace{0mm}\left\{ \begin{array}{ll} 
{\rm (a)\ } \mathcal{F}:\Omega\times \mathbb{S}^d\to \mathbb{S}^d.\\[2mm]
{\rm (b)\ } {\rm There\ exists\ } L_{\mathcal{F}}>0 {\rm \ such\ that\ }\\[0mm]
\qquad \|\mathcal{F}(\by,\btau_1)-\mathcal{F}(\by,\btau_2)\|\leq L_\mathcal{F} \left( \|\btau_1-\btau_2\|\right)\\[0mm]
\quad \qquad {\rm for\ all\ } \btau_1,\btau_2\in\mathbb{S}^d, {\rm \ a.e.} \ \by\in\Omega.\\[2mm]
{\rm (c)\ } \by\mapsto\mathcal{F}(\by,\btau) {\rm  \ is\ measurable\  on\ }\Omega,\  \mbox{for any}\ \, \btau\in\mathbb{S}^d.\hspace{5mm}\\[2mm]
{\rm (d)\ } \by\mapsto\mathcal{F}(\by,\bcero) \in Q.
\label{Fm}
\end{array}\right.
\end{eqnarray}

{\color{black}
\begin{eqnarray}
&&\hspace{0mm}\left\{ \begin{array}{ll} 
{\rm (a)\ } h:\Gamma_3 \times \mathbb{R}\times \mathbb{R} \to\mathbb{R}_+.\\[2mm]
{\rm (b)\ } {\rm There\ exists\ } L_h>0 {\rm \ such\ that\ }\\[0mm]
\qquad |h(\by,r_1,s_1)-h(\by,r_2,s_2)|\leq L_h \big(|r_1-r_2 |+|s_1-s_2 |)\\[0mm]
\qquad \qquad {\rm for\ all\ } r_1,\,r_2,\, s_1,\,s_2 \in\mathbb{R}, {\rm \ a.e.\ }\by \in \Gamma_3.\\[2mm]
{\rm (c)\ }\ \by\mapsto h(\by,r,s) {\rm  \ is\ measurable\ on\ }\Gamma_3,\  {\rm  for\ any\ }r,\, s\in\mathbb{R}.\\[2mm]
{\rm (e)\ }\by\mapsto h(\by,0,0)\in L^2(\Gamma_3).
\label{h}
\end{array}\right.
\end{eqnarray}}
\begin{eqnarray}
&&\hspace{0mm}\label{67}
\bff_0\in C(I;L^2(\Omega)^d),\qquad\qquad \bff_2\in C(I;L^2(\Gamma_2)^d),\\ [2mm]
&& \hspace{0mm}\beta\in L^\infty(\Omega).\label{be}\\ [2mm]
&& \hspace{0mm}k\in L^\infty(\Gamma_3),\quad k(\by)\ge 0\ \ 
{\rm \ a.e.\ \ }\by\in\Gamma_3.\label{p}\\ [2mm]
&&\hspace{0mm}\bu_0\in V,\qquad \bsigma_0\in Q.\label{Reg} \\[2mm]
&&\hspace{0mm} {\color{black} \left\{\begin{array}{l}
{\rm  There\ exist}\  G\in H^2(\Omega)\ {\rm and}\ M_0,\, M_1\in\mathbb{R}\ \mbox{such that}\\
 \	g=\gamma_0(G)\;{\rm on }\; \Gamma_3 \: {\rm and}\; 0<M_0\leq G(\by)\leq M_1\; \mbox{for all}\ \by\in \overline{\Omega}.\end{array}\right.}\label{gg}
\end{eqnarray}

{\color{black} Note that in (\ref{gg}) and below $\gamma_0:H^1(\Omega)\to L^2(\Gamma)$ denotes the trace operator. Moreover, note that the condition (\ref{gg}) make sense since $d\in\{2,3\}$ and, therefore, $H^2(\Omega)\subset C(\overline{\Omega})$.}

We turn in what follows to the variational analysis of Problem  $\mathcal{P}^{ve}$ and, to this end, besides the function $\xi:\Gamma_3\times I\to\mathbb{R}$ defined by
(\ref{csi}), we consider the irreversible stress field $\bsigma^{ir}:\Omega\times I\to\mathbb{S}^d$ and the set of {\color{black} admissible displacements} fields $K\subset V$ defined by
\begin{eqnarray}
&&\label{sir}
\bsigma^{ir}=\bsigma-\mathcal{E}\bepsilon(\bu),\\
&&K=\lbrace\, \bv\in V: v_\nu\leq g\; {\rm \ a.e.\ on}\ \Gamma_3\,\rbrace.\label{Km}
\end{eqnarray}
Then, using standard arguments we deduce
the following variational formulation of the problem. 

\medskip

\noindent \textbf{Problem $\mathcal{P}^{ve}_V$}. 
{\it Find an irreversible stress field $\bsigma^{ir}:I\to Q$, an accumulated penetration
function $\xi:I\to L^2(\Gamma_3)$ and a displacement field $\bu:I\to V$ such that}
\begin{align}
&\dot{\bsigma}^{ir}(t)=
\beta(\mathcal{E}\bepsilon(\bu(t))  +\bsigma^{ir}(t)-\mathcal{F}\bepsilon(\bu(t))),\quad\dot{\xi}(t)=u_\nu^+(t)\quad \forall\, t\in I, \label{Diffeq-Pve}\\
&\bsigma^{ir}(0)=\bsigma_0-\mathcal{E}\bepsilon(\bu_0),\quad \xi(0)=0,\label{InitCond-Pve}\\ 
&\bu(t)\in K, \quad   
\int_{\Omega}(\mathcal{E}\bepsilon(\bu(t))+\bsigma^{ir}(t))\cdot(\bepsilon(\bv)-\bepsilon(\bu(t))  \,dy\nonumber \\
& \qquad\qquad +\int_{\Gamma_3} ku_\nu^+(t)(v_\nu-u_\nu(t))\,da+
{\color{black}
\int_{\Gamma_3} h(\xi(t),u_\nu^+(t))(v_\nu^+-u_\nu^+(t))\,da}
\nonumber\\
&\qquad\qquad\geq 
\int_\Omega \bff_0(t)\cdot \bv-\bu(t)\,  dy+\int_{\Gamma_2} \bff_2(t)\cdot(\bv -\bu(t)\,da
\quad \forall\, \bv\in K,\ \, t\in I.  \label{VarIneq-Pve} 
\end{align}

\medskip
The unique solvability of Problem  $\mathcal{P}^{ve}_V$ is provided by the following existence and uniqueness result.

\begin{teo}\label{t8}
Assume $(\ref{E})$--$(\ref{gg})$. Then Problem $\mathcal{P}^{ve}_V$  has a unique solution which satisfies $\bsigma^{ir}\in C^1(I;Q)$, $\xi\in C^1(I;L^2(\Gamma_3))$, $\bu\in C(I;V)$.
\end{teo}

\begin{proof} We consider the product spaces $X=Q\times L^2(\Gamma_3)$ and $Z=L^2(\Omega)^d\times L^2(\Gamma_2)^d$  endowed with the canonical inner products $(\cdot,\cdot)_X$ and  $(\cdot,\cdot)_Z$, respectively, as well as 
the operators $F:I\times X\times V\to X$,  $A:X\times V\to V^{*}$, $\pi:V\to Z$ and  the functions $j:X\times V\times V\to\mathbb{R}$, $\bff:I\to V^*$ given by
\begin{align}
&F(t,\bx,\bu )=\Big(\beta\big(\mathcal{E}\bepsilon(\bu))+\bsigma-\mathcal{F} \bepsilon(\bu)\big),u_\nu^+\Big),\label{F-Ejem} \\ 
& \la A(\bx,\bu),\bv\ra =\left(\mathcal{E}\bepsilon(\bu)+\bsigma,\bepsilon(\bv) \right)_Q+\int_{\Gamma_3} ku_\nu^+v_\nu \; da,\label{A-Ejem}\\
& \pi\bv=(\bv,\bv|_{\Gamma_2}),\label{pi-Ejem}\\
&{\color{black}j(\bx,\bu,\bv)=\int_{\Gamma_3}h(\xi,u^+_\nu)v_\nu^+\,da,}\label{jem}\\
& (\bff (t),\bz)_Z=\int_\Omega \bff_0(t)\cdot \bz_1\;  dx+\int_{\Gamma_2} \bff_2(t)\cdot \bz_2 \; da\label{fem}
\end{align}
for all $t\in I$, $\bx=(\bsigma,\xi)\in X$, $\bu,\bv\in V$, $\bz=(\bz_1,\bz_2)\in Z$. Note that in (\ref{pi-Ejem}) notation $\bv|_{\Gamma_2}$ represents the trace of $\bv$ in $\Gamma_2$.
Moreover, consider the element of $X$ given by
\begin{equation}\label{100}
\bx_0=(\bsigma_0-\mathcal{E}\bepsilon(\bu_0),0).
\end{equation}
Then, it is easy to see that
Problem $\mathcal{P}^{ve}_V$ is equivalent to the problem of finding two functions $\bx=(\bsigma^{ir},\xi):I\to X$ and  $\bu:I\to V$ such that
 \begin{align}
&\dot{\bx}(t)=F(t,\bx(t),\bu(t))\qquad \forall\, t\in I, \label{1mm}\\
&\bx(0)=\bx_0,\label{2mm}\\
& \bu(t)\in K, \qquad  \langle  A(\bx(t),\bu(t)),\bv-\bu(t)\rangle+j(\bx(t),\bu(t),\bv)-j(\bx(t),\bu(t),\bu(t)) \nonumber \\
&\qquad \qquad\qquad\qquad \geq \la \bff(t),\pi \bv-\pi\bu(t)\ra_Z\quad\ \forall\, \bv\in K,\ t\in I.  \label{3mm}
\end{align}

Remark that, with the previous notation, all the conditions in Theorem \ref{ExistUniq-P} are satisfied for the differential variational inequality (\ref{1mm})--(\ref{3mm}) For instance, it is easy to see that assumptions (\ref{E}), (\ref{Fm}), (\ref{be}) and  (\ref{p})  imply that the operators (\ref{F-Ejem}) and (\ref{A-Ejem}) satisfy conditions (\ref{F}) and (\ref{A}), respectively, the later with $m=m_{\cal E}$. Moreover, the regularity $(\ref{Reg})$ and $(\ref{67})$ imply that $(\ref{x0})$ and 
$(\ref{f})$ hold, too. {\color{black}In addition,  assumption (\ref{gg}) combined with standard arguments implies that the set (\ref{Km})
satisfies condition (\ref{K}) and,}
using the
assumption (\ref{h}) and the Sobolev trace inequality  it is easy to see that condition (\ref{j}) holds with $\beta=0$. To conclude, we deduce from Theorem \ref{ExistUniq-P} the existence of a unique solution $\bx=(\bsigma^{ir},\xi)\in C^1(I;X)$, $\bu\in C(I;V)$ which satisfies   (\ref{1mm})--(\ref{3mm}). Then, using the equivalence between Problem
$\mathcal{P}^{ve}_V$ and the differential quasivariational inequality (\ref{1mm})--(\ref{3mm}), we deduce that
$(\bsigma^{ir},\xi,\bu)$ is the unique solution to Problem
$\mathcal{P}^{ve}_V$ with regularity $\bsigma^{ir}\in C^1(I;Q)$, $\xi\in C^1(I;L^2(\Gamma_3))$, $\bu\in C(I;V)$,
which concludes the proof.
\end{proof}

We now study the continuous dependence of the solution to Problem $\mathcal{P}^{ve}_V$ with respect to the data.  Various cases can be consideblack and various convergence results can be obtained, based on Theorem \ref{t1}.   Here, for simplicity, we restrict ourselves to provide only one example, which concerns the dependence of the solution with respect to the density of surface tractions and the thickness $g$. Therefore, we assume in what follows that
$(\ref{E})$--$(\ref{gg})$ hold and, moreover,
we assume that there exists two functions $\theta$ and $\widetilde{\bff}_{2}$ such that 
\begin{eqnarray}
&&\label{x1}\theta\in C(I;\mathbb{R}),\quad
\widetilde{\bff}_2\in L^2(\Gamma_2)^d,\\[2mm]
&&\label{x2}\bff_{2}(t)=\theta(t)\widetilde{\bff}_{2}\qquad\forall\, t\in I.
\end{eqnarray}
In addition, for each $n\in\mathbb{N}$ we  consider a perturbation  $\bff_{2n}$ and {\color{black}$g_n=\gamma_0(G_n)$} of $\bff_2$ and {\color{black}$g=\gamma_0(G)$}, respectively, such that
\begin{eqnarray} 
&&\label{x3}\bff_{2n}(t)=\theta(t)\widetilde{\bff}_{2n}\qquad\forall\, t\in I\quad{\rm with}\quad \widetilde{\bff}_2\in L^2(\Gamma_2)^d.\\[2mm]
&&\hspace{0mm} {\color{black} 
	G_n\in H^{2}(\Omega)\ \ {\rm and }\ \ 0<M_0\leq G_n(\by)\leq M_1\ \ \mbox{for all}\ \by\in \overline\Omega}\label{x4}.\\ [2mm]
&&\label{x5}\widetilde{\bff}_{2n}(t)\rightharpoonup\bff_{2}(t) \quad{\rm in}\ \ L^2(\Gamma_2)^d\quad\forall\, t\in I.\\ [2mm]
&&\label{x6}{\color{black} G_n\rightharpoonup G \quad{\rm in}\ \  H^{2}(\Omega).}
\end{eqnarray}
For each $n\in\mathbb{N}$ we consider
Problem  $\mathcal{P}^{ve}_{Vn}$ obtained by replacing in Problem $\mathcal{P}^{ve}_{V}$ the data $\bff_2$ and $g$ with $\bff_{2n}$ and $g_n$, respectively.
Then, Theorem \ref{t8} guarantees that 
$\mathcal{P}^{ve}_{Vn}$ has a unique solution $(\bsigma_n^{ir},\xi_n,\bu_n)$, with regularity $\bsigma^{ir}_n\in C(I;Q)$, $\xi_n\in C(I;L^2(\Gamma_3)$,
$\bu_n\in C(I;V)$. Moreover we have the following convergence result.

\begin{teo}\label{t5}
	Assume $(\ref{E})$--$(\ref{gg})$ and $(\ref{x1})$--$(\ref{x6})$. Then, the  solution $(\bsigma_n^{ir},\xi_n,\bu_n)$ of Problem $\mathcal{P}^{ve}_{Vn}$ converges   to the solution $(\bsigma^{ir},\xi,\bu)$ of Problem $\mathcal{P}^{ve}_{V}$  as $n\to\infty$, i.e., for each $t\in I$ we have
	\begin{equation*}
	\bsigma_n^{ir}(t)\to\bsigma^{ir}(t)
	\ \ {\rm in\ } Q,\quad \xi_n(t)\to\xi(t)
	\ \ {\rm in\ } L^2(\Gamma_3),\quad 
	\bu_n(t)\to\bu(t)
	\ \  {\rm in\ } V\ \ {\rm as}\ n\to\infty.
	\end{equation*}
\end{teo}

\begin{proof}
First, we remark that  the set of constraints associated to Problem $\mathcal{P}^{ve}_{Vn}$ is given by
\begin{equation}\label{Kmn}
K_n=\lbrace \bv\in V: v_\nu\leq g_n\; {\rm \ a.e.\ on\ }\Gamma_3\rbrace.
\end{equation}

 {\color{black} Let $v\in K$. Then, 
 assumptions (\ref{gg}) and (\ref{x4}) allow us consider the sequence $\{v_n\}\subset V$ defined by
 $v_n=\frac{G_n}{G}\,v$, for each $n\in\mathbb{N}$ 
 We now use definitions (\ref{Km}),  (\ref{Kmn}) 
 and equalities  $g_n=\gamma_0(G_n)$, $g=\gamma_0(G)$ to see that $v_n\in K_n$ for each $n\in\mathbb{N}$. Moreover, using (\ref{x4}),  (\ref{x6}) and the compactness of the inclusion $H^2(\Omega)\subset H^1(\Omega)$  (see, for instance \cite{Ad1975}) it is easy to see that
 $v_n\to v$ in $V$. We conclude from here that condition (\ref{Kn-Conv})(a) is satisfied. 
 
 Assume now that $\lbrace v_n\rbrace$  is a sequence of elements of $V$ such that
$v_n\in K_n$ for all $n\in\mathbb{N}$ and  $v_n\wl v$ in $V$. Then, 
\begin{equation}\label{56}
 v_{n\nu}\leq g_n\;\ \  {\rm \ a.e.\ on\ }\Gamma_3,\ \ \mbox{for all $n\in\mathbb{N}$.}
\end{equation}
 Moreover, compactness arguments guarantee that the convergences
$v_n\wl v$ in $V$ and $G_n\wl G$ in $H^{2}(\Omega)$ imply that
$v_{n\nu}\to v_\nu$ and  $g_n\to g$, both in $L^2(\Gamma_3)$.
Therefore, passing to some subsequences, again denoted by
 $\lbrace v_n\rbrace$ and  $\lbrace g_n\rbrace$,
 we can assume that
 \begin{equation}\label{57}
  v_{n\nu}\to v_\nu,\quad  g_n\to g\quad  {\rm \ a.e.\ on\ }\Gamma_3.
 \end{equation}
It follows now from (\ref{56}) and (\ref{57}) that
$v_{\nu}\leq g\   {\rm \ a.e.\ on\ }\Gamma_3$
which shows that {\color{black}$v\in K$} and, hence, 
(\ref{Kn-Conv})(b) holds.
The proof of Theorem \ref{t5} is now a direct consequence of Theorem \ref{t1} and Remark \ref{rem1}.}
\end{proof}

We now turn to the optimal control of 
Problem $\mathcal{P}^{ve}_{V}$ and, to this end, we shall use Theorem \ref{t2}. For simplicity, we restrict ourselves to provide the following example. 

Assume  that
$(\ref{E})$--$(\ref{gg})$ hold
and denote by $W$ the product space  {\color{black} 
$W=L^2(\Gamma_2)^d\times H^{2}(\Omega)$ endowed with the canonical Hilbertian structure.
Moreover, consider the set $U\subset W$ defined by
\begin{eqnarray}
&&\label{UU}
U=\lbrace\, q=(\widetilde{\bff}_2,{G})\in W\ :\ \|\widetilde{\bff}_2\|_{L^2(\Gamma_2)^d}\leq M_2\ ,\  \|G\|_{H^2(\Omega)}\leq M_3,\\
&&\qquad\quad \ M_0\leq{G}(\by)\leq M_1\ \ {\rm for\ all}\ \by\in\overline\Omega \,\rbrace\nonumber
\end{eqnarray}
where $M_0$,  $M_1$, $M_2$ and $M_3$ are given positive constants such that $M_0\le M_1$ and $M_3\ge M_0(mes(\Omega)^{\frac{1}{2}}$. Note that  the set $U$ is nonempty since, for instance, $(\bzero_{L^2(\Gamma_2)^d},M_0)\in U.$} For any {\color{black}$q=(\widetilde{\bff}_2,{G})\in U$} we consider
Problem  $\mathcal{P}^{ve}_{Vq}$ obtained by replacing in Problem $\mathcal{P}^{ve}_{V}$ the data $\bff_2$ and {\color{black}$g$} with  
$\bff_{2q}$ and {\color{black}$g_q$}, respectively, where
\begin{eqnarray*} 
\bff_{2q}(t)=\theta(t)\widetilde{\bff}_{2}\qquad\forall\, t\in I,\qquad  {\color{black}g_q=\gamma_0({G})}
\end{eqnarray*}
and $\theta\in C(I;\mathbb{R})$. Then, Theorem \ref{t8} guarantees that 
$\mathcal{P}^{ve}_{Vq}$ has a unique solution $(\bsigma_q^{ir},\xi_q,\bu_q)$, we regularity $\bsigma^{ir}_q\in C^1(I;Q)$, $\xi_q\in C^1(I;L^2(\Gamma_3)$, $\bu_q\in C(I;V)$. 
Consider now the following optimal control problem in which, for any $q\in U$, $u_{q\nu}$ represents the normal component of the function $\bu_q$.

\medskip

\noindent \textbf{Problem $\mathcal{Q}^{ve}$}. 
{\it  Given $t\in I$ and $\phi\in L^2(\Gamma_3)$, find $q^*=(\bff_2^*,{\color{black}G^*})\in U$ such that }
\begin{align}\label{Problem:Q-Ejem}
\int_{\Gamma_3} |u_{q^*\nu}(t)-\phi|^2\,da\le\int_{\Gamma_3} |u_{q\nu}(t)-\phi|^2\,da\qquad\forall\, q\,\in U.
\end{align}

We have the following existence result.

\begin{teo}\label{t6}
Under the previous assumptions,
the optimal control problem $\mathcal{Q}^{ve}$ has at least one solution $q^*=(\bff_2^*,{\color{black}G^*})\in U$.
\end{teo}

\begin{proof}
It is easy to see that the set (\ref{UU}) satisfies condition (\ref{U}) and (\ref{Ubound}) on the space $W=L^2(\Gamma_2)^d\times {\color{black}H^2(\Omega)}$. Moreover, the function
${\cal L}:X\times V\times U\to\mathbb{R}$ defined by
\[{\cal L}(\bx,\bu,q)= \int_{\Gamma_3} |u_{q\nu}(t)-\phi|^2\,da\qquad\forall\,\bx\in X,\ \bu\in V, q\,\in U\]
satisfies condition  (\ref{LimInfL}) with $X=Q\times L^2(\Gamma_3)$. We now use Theorem \ref{t2} and Remark \ref{rem2} to conclude the proof.
\end{proof}
We end this section with some comments and mechanical interpretation of our results. First, the variational formulation $\mathcal{P}^{ve}_V$ of Problem $\mathcal{P}^{ve}$, in terms of the irreversible stress, accumulated penetration and displacement field, is new and nonstandard. Nevertheless, we refer to solution $(\bsigma^{ir},\xi,\bu)$ of $\mathcal{P}^{ve}_V$ as the weak solution of the frictionless contact problem $\mathcal{P}^{ve}$. Therefore, Theorem \ref{t8} provides the unique solvability of this viscoelastic contact problem. Next, Theorem \ref{t5} shows that the weak solution depends continuously on the density of surface tractions and the {\color{black}thickness} of the rigid-elastic layer. Finally, the
mechanical interpretation of the optimal control problem $\mathcal{Q}^{ve}$ is the following:
given a contact process of the form $(\ref{S58})$--$(\ref{S64})$,  $(\ref{x1})$, 
$(\ref{x2})$ and a time moment $t\in I$,
we are looking for a pair $q=(\widetilde{\bff}_2^*,{\color{black}G^*})\in U$
such that
the corresponding penetration of the viscoelastic body  at $t$ is as close as possible to the ``desiblack  penetration" $\phi$.
Theorem $\ref{t6}$ guarantees the existence of at least one solution to this problem.

\section*{Acknowledgements}
This work   has received funding from the European Union's Horizon 2020 Research
and Innovation Programme under the Marie Sklodowska-Curie Grant Agreement No
823731 CONMECH. It has also been partially supported by the Project PIP No. 0275 from CONICET-UA, Rosario, Argentina. {\color{black}The authors would like to thank the anonymous referee whose insightful comments have benefit the contents of this article.}

\noindent

\bibliographystyle{elsarticle-num} 
\bibliography{Biblio}

\begin{thebibliography}{10}
\expandafter\ifx\csname url\endcsname\relax
  \def\url#1{\texttt{#1}}\fi
\expandafter\ifx\csname urlprefix\endcsname\relax\def\urlprefix{URL }\fi
\expandafter\ifx\csname href\endcsname\relax
  \def\href#1#2{#2} \def\path#1{#1}\fi

\bibitem{St1964}
G.~Stampacchia, Formes bilin\'eaires coercitives sur les ensembles convexes,
  C.R. Acad. Sc. Paris 258 A (1964) 4413--4416.

\bibitem{LiSt1967}
J.~L. Lions, G.~Stampacchia, Variational inequalities, Comm. Pure Appl. Math.
  20 (1967) 493--519.

\bibitem{Br1972}
H.~Brezis, Probl\`emes unilat\'eraux, J. Math. Pures et Appl. 51 (1972) 1--168.

\bibitem{Ne2012}
{\color{black}J}.~{\color{black}Nec\v{a}s}, {\color{black}Direct methods in the
  theory of elliptic equations}, {\color{black}Springer},
  {\color{black}Berlin}, {\color{black}2012}.

\bibitem{Li1969}
J.~L. Lions, Quelques m\'ethodes de r\'esolution des probl\`emes aux limites
  non lin\'eaires, Dunod-Gauthier Villars, Paris, 1969.

\bibitem{EkTe1973}
{\color{black}I}.~{\color{black}Ekeland},
  {\color{black}R}.~{\color{black}Temam}, {\color{black}Convex Analysis and
  Variational problems}, {\color{black}North-Holland},
  {\color{black}Amsterdam-Oxford}, {\color{black}1976}.

\bibitem{BaCa1984}
{\color{black}C}.~{\color{black}Baiocchi},
  {\color{black}A}.~{\color{black}Capelo}, {\color{black}Variational and
  quasivariational inequalities- Applications to free boundary problems},
  {\color{black}John Wiley}, {\color{black} New York}, {\color{black}1984}.

\bibitem{KiSt1980}
D.~Kinderlehrer, G.~Stampacchia, An introduction to variational inequalities
  and their applications, Academic Press, New York, 1980.

\bibitem{Cr1984}
J.~Crank, Free and moving boundary problems, Clarendon Press, Oxford, 1984.

\bibitem{Ro1987}
J.~F. Rodriguez, Obstacle problems in mathematical physics, North-Holland,
  Amsterdam, 1987.

\bibitem{AlSo1993}
V.~Alexiades, A.~D. Solomon, Mathematical modeling of melting and freezing
  processes, Hemisphere Publishing Corp., Washington, 1993.

\bibitem{DeMiPa2003}
Z.~Denkowski, S.~Mig\'orski, N.~S. Papageorgiou, An introduction to nonlinear
  analysis: {T}heory, Springer, New York, 2003.

\bibitem{Gw2003}
{\color{black}J}.~{\color{black}Gwinner}, {\color{black}Time dependent
  variational inequalities - some recent trends, in: P. Daniele, A. Maugeri
  (Eds.), Equilibrium Problems and Variational Models},
  {\color{black}Springer-Verlag}, {\color{black}New York}, {\color{black}2003}.

\bibitem{Ba1971}
C.~Baiocchi, Sur un probl\`eme \`a fronti\`ere libre traduisant le filtrage de
  liquides \`a travers des milieux poreux, C.R. Acad. Sc. Paris 273 A (1971)
  1215--1217.

\bibitem{Du1973}
G.~Duvaut, R\'esolution d\textsc{\char13} un prob\`eme de {S}tefan (fusion
  d\textsc{\char13} un bloc de glace \`a z\'ero degr\'e), C. R. Acad. Sci.
  Paris 276 A (1973) 1461--1463.

\bibitem{Ta1979}
D.~A. Tarzia, Sur le probl\`eme de {S}tefan \`a deux phases, C.R. Acad. Sci.
  Paris 288 A (1979) 941--944.

\bibitem{DuLi1976}
{\color{black}G}.~{\color{black}Duvaut}, {\color{black}J.
  L}.~{\color{black}Lions}, {\color{black}Inequalities in Mechanics and
  Physics}, {\color{black} Springer-Verlag}, {\color{black}Berlin},
  {\color{black}1976}.

\bibitem{Pa1985}
P.~Panagiotopoulos, Inequality problems in mechanics and applications,
  Birkh\"{a}user, Basel, 1985.

\bibitem{SoHaSh2006}
M.~Sofonea, W.~Han, M.~Shillor, Analysis and approximation of contact problems
  with adhesion or damage, Chapman \& Hall/CRC Boca Raton, 2006.

\bibitem{SoMa2012}
M.~Sofonea, A.~Matei, Mathematical {M}odels in {C}ontact {M}echanics, Cambridge
  University Press, Cambridge, 2012.

\bibitem{MiOcSo2013}
S.~Mig\'orski, A.~Ochal, M.~Sofonea, Nonlinear Inclusions and Hemivariational
  Inequalities. Models and analysis of contract problems, Springer, New York,
  2013.

\bibitem{SoMi2018}
M.~Sofonea, S.~Mig\'orski, Variational-hemivariational inequalities with
  applications, Boca Raton: Chapman {\&} Hall/CRC Press, 2018.

\bibitem{Li1968}
J.~L. Lions, Contr\^ole optimal des syst\`emes gouvern\'es par des \'equations
  aux d\'eri\'ees partielles, Dunod-Gauthier Villars, Paris, 1968.

\bibitem{NeSpTi2006}
P.~Neittaanmaki, J.~Sprekels, D.~Tiba, Optimization of elliptic systems. Theory
  and applications, Springer, New York, 2006.

\bibitem{HiPiUlUl2009}
M.~Hinze, P.~Pinnau, M.~Ulbrich, S.~Ulbrich, Optimization with PDE constraints,
  Springer, New York, 2009.

\bibitem{Tr2010}
F.~Tr\"oltzsch, Optimal control of partial differential equations. Theory,
  methods and applications, American Mathematical Society, Providence, 2010.

\bibitem{Cl2013}
F.~Clarke, Functional analysis, calculus of variations and optimal control,
  Springer, London, 2013.

\bibitem{AuCe1984}
J.~P. Aubin, A.~Cellina, Differential inclusions, Springer-Verlag, New York,
  1984.

\bibitem{Gw2013}
J.~Gwinner, Three-field modelling of nonlinear nonsmooth boundary problems and
  stability of differential mixed variational inequalities, Abstract and
  Applied Analysis, Article ID 108043,$ $ (2013) 1--10.

\bibitem{LiZeMo2016}
Z.~H. Liu, S.~D. Zeng, D.~Motreanu, Evolutionary problems driven by variational
  inequalities, J. Differential Equations 260 (2016) 6787--6799.

\bibitem{LiMiZe2017}
Z.~H. Liu, S.~Mig\'orski, S.~D. Zeng, Partial differential variational
  inequalities involving nonlocal boundary conditions in banach spaces, J.
  Differential Equations 263 (2017) 3989--4006.

\bibitem{LiZe2019}
Z.~H. Liu, S.~Zeng, Penalty method for a class of differential variational
  inequalities, Applicable Analysis. $ $ (2019).
\newblock \href {https://doi.org/10.1080/00036811.2019.1652736}
  {\path{doi:10.1080/00036811.2019.1652736}}.

\bibitem{Gw2007}
{\color{black}J}.~{\color{black}Gwinner}, {\color{black}On differential
  variational inequalities and projected dynamical systems - equivalence and a
  stability result}, {\color{black}Discrete and Continuous, Dynamical Systems,
  Supplement} {\color{black}2007} ({\color{black}2007})
  {\color{black}467--476}.

\bibitem{Gw2013-2}
{\color{black}J}.~{\color{black}Gwinner}, {\color{black}On a new class of
  differential variational inequalities and a stability result},
  {\color{black}Math. {P}rogram. } {\color{black}Ser. B 139}
  ({\color{black}2013}) {\color{black}205}--{\color{black}221}.

\bibitem{LiSo2018}
Z.~H. Liu, M.~Sofonea, Differential quasivariational inequalities in contact
  mechanics, Mathematics and Mechanics of Solids 24 (2018) 845--861.

\bibitem{BlOe1994}
{\color{black}E}.~{\color{black}Blum}, {\color{black}W}.~{\color{black}Oettli},
  {\color{black}From optimization and variational inequalities to equilibrium
  problems}, {\color{black}Math. {S}tudent} {\color{black}63}
  ({\color{black}1994}) {\color{black}123}--{\color{black}145}.

\bibitem{ChChRi1999}
{\color{black}0}.~{\color{black}Chadli},
  {\color{black}Z}.~{\color{black}Chbani},
  {\color{black}H}.~{\color{black}Riahi}, {\color{black}Recession methods for
  equilibrium problems and applications to variational and hemivariational
  inequalities}, {\color{black}Discrete and Continuous Dynamical Systems-A}
  {\color{black}5} ({\color{black}1999}) {\color{black}185--196}.

\bibitem{ChChRi2000}
{\color{black}0}.~{\color{black}Chadli},
  {\color{black}Z}.~{\color{black}Chbani},
  {\color{black}H}.~{\color{black}Riahi}, {\color{black}Equilibrium problems
  with generalized monotone bifunctions and applications to variational
  inequalities}, {\color{black}Journal of Optimization Theory and Applications}
  {\color{black}105} ({\color{black}2000}) {\color{black}299--323}.

\bibitem{DaGiMa2003}
{\color{black}P}.~{\color{black}Daniele},
  {\color{black}F}.~{\color{black}Giannessi},
  {\color{black}A}.~{\color{black}Maugeri (Eds.)}, {\color{black}Equilibrium
  Problems and Variational Models}, {\color{black}Springer-Verlag},
  {\color{black}New York}, {\color{black}2003}.

\bibitem{Mo1969}
U.~Mosco, Convergence of convex sets and of solutions of variational
  inequalities, Adv. Math. 3 (1969) 510--585.

\bibitem{Ad1975}
{\color{black} R. A}.~{\color{black}Adams}, {\color{black}Sobolev Spaces},
  {\color{black}Academic Press}, {\color{black}New York}, {\color{black}1975}.

\end{thebibliography}

\end{document}